\documentclass[11pt]{article}
\usepackage[utf8]{inputenc}
\usepackage{authblk}

\usepackage[letterpaper,margin=1in]{geometry}

\usepackage[colorlinks]{hyperref}

\usepackage{amssymb,amsthm,amscd,latexsym,mathrsfs,units,enumerate,bm,bbm,cancel,physics,mathtools, braket}
\usepackage{color,xcolor}
\usepackage{graphicx,wrapfig}
\usepackage{hieroglf}
\usepackage{epic,eepic}
\usepackage{todonotes}

\graphicspath{{figs/}}

\definecolor{skyblue}{rgb}{0.85,0.85,1}

\renewcommand{\geq}{\geqslant}
\renewcommand{\leq}{\leqslant}

\newcommand{\w}{\omega}

\DeclareMathOperator{\disc}{disc}

\DeclareMathOperator{\res}{res}

\theoremstyle{plain}
\newtheorem{theorem}{Theorem}

\newtheorem{lemma}[theorem]{Lemma}

\numberwithin{equation}{section}

\theoremstyle{definition}

\newtheorem*{remark*}{Remark}
\newtheorem*{definition*}{Definition}

\title{Floquet theory and stability analysis for Hamiltonian PDEs}
\author[1]{Jared~C.~Bronski\thanks{E-mail:~bronski@illinois.edu. JCB would like to thank the University of Sydney and the Sydney Mathematical Research Institute for support during the writing of this paper.}}
\author[1]{Vera~Mikyoung~Hur\thanks{E-mail:~verahur@illinois.edu. VMH is supported by NSF DMS--2009981.}}
\author[2]{Robert~Marangell\thanks{E-mail:~robert.marangell@sydney.edu.au. RM is partially supported by ARC Discovery Project DP210101102.}}
\affil[1]{Department of Mathematics, University of Illinois Urbana-Champaign \protect\\ 
Urbana, IL 61801, USA}
\affil[2]{School of Mathematics and Statistics, University of Sydney \protect\\ Sydney, NSW 2006, Australia}

\begin{document}

\maketitle

\abstract{We analyze Floquet theory as it applies to the stability and instability of periodic traveling waves in Hamiltonian PDEs. Our investigation focuses on several examples of such PDEs, including the generalized KdV and BBM equations (third order), the nonlinear Schr\"odinger and Boussinesq equations (fourth order), and the Kawahara equation (fifth order).
Our analysis reveals that the characteristic polynomial of the monodromy matrix inherits symmetry from the underlying PDE, enabling us to determine the essential spectrum along the imaginary axis and bifurcations of the spectrum away from the axis, employing the Floquet discriminant. We present numerical evidence to support our analytical findings.}

\tableofcontents

\section{Introduction}\label{sec:intro}

The subject of the investigation here is the stability and instability of periodic traveling waves in Hamiltonian partial differential equations (PDEs) in one spatial dimension. This entails addressing spectral problems of the form
\begin{equation}\label{eqn:JH}
\lambda v=\boldsymbol{\mathcal{J}}\boldsymbol{\mathcal{L}},\qquad \lambda\in\mathbb{C},
\end{equation}
where $\boldsymbol{\mathcal{J}}$ represents a symplectic form, $\boldsymbol{\mathcal{L}}$ is a linear self-adjoint operator corresponding to the second variation of an appropriate Hamiltonian, and $\boldsymbol{\mathcal{J}}\boldsymbol{\mathcal{L}}$ has periodic coefficients. It is well-established that the $L^2(\mathbb{R})$ essential spectrum of such an operator remains invariant under the transformations
\[
\lambda \mapsto -\lambda\quad\text{and}\quad\lambda \mapsto \overline{\lambda},
\]
implying that the spectrum is symmetric with respect to reflections across the real and imaginary axes. Our emphasis lies in quasi-periodic eigenvalue problems that exhibit symmetry of the kind. Prior research on this subject can be found in \cite{JMMP1,JMMP2,JMMP3,BR1}, among many others, and \cite{MM} in particular. See also \cite{Panos_2023} for some similar, related results. 

We can reformulate \eqref{eqn:JH} as
\begin{equation}\label{eqn:A}
\mathbf{v}_x={\bf A}(x,\lambda)\mathbf{v},
\end{equation}
where $\mathbf{A}(x,\lambda)$ is an $n\times n$ matrix-valued function of $x\in\mathbb{R}$, satisfying $\mathbf{A}(x+T,\lambda)=\mathbf{A}(x,\lambda)$ for some $T>0$, the period. Throughout, we assume {\em generalized Hamiltonian symmetry}. That is,
\begin{itemize}
\item[(A1)] $\mathbf{A}(x,\lambda)$ is real for $\lambda\in\mathbb{R}$.
\item[(A2)] ${\bf A}^\top(x,\lambda){\bf B}(\lambda)=-{\bf B}(\lambda){\bf A}(x,-\lambda)$ for some matrix $\mathbf{B}(\lambda)$, independent of $x$, and nonsingular everywhere except for at most finitely many values of $\lambda$.  
\end{itemize}
These assumptions can be somewhat relaxed---for instance, for the generalized BBM equation---but they encompass the majority of relevant examples. Importantly, $n$ does not have to be even. Here we present examples for $n=3$ (such as the generalized KdV and BBM equations) and $n=5$ (such as the fifth-order KdV or Kawahara equation). Additionally, we remark that \text{(A2)} is weaker than the infinitesimal symplecticity assumption \cite{Jones1988,StarYak1975}. 
Among all the examples discussed herein, only the linearizations of the nonlinear Schr\"odinger equation about a trivial phase solution and the Boussinesq equation about a stationary solution satisfy infinitesimal symplecticity. For $n$ odd, infinitesimal symplecticity implies that $\mathbf{A}(x,\lambda)$ must be singular, but generalized Hamiltonian symmetry does not. 

We define the {\em monodromy matrix} of \eqref{eqn:A} as
\begin{equation}\label{def:M}
{\bf M}(\lambda)={\bf V}(T,\lambda),\quad\text{where}\quad
{\bf V}_x={\bf A}(x,\lambda){\bf V}\quad\text{and}\quad{\bf V}(0,\lambda)={\bf I}_n.
\end{equation}
Here $\mathbf{I}_n$ represents the $n\times n$ identity matrix. We define the {\em characteristic polynomial} of the monodromy matrix of \eqref{eqn:A} as
\begin{equation}\label{def:p}
p(\mu,\lambda)=\det(\mathbf{M}(\lambda)-\mu\mathbf{I}_n).
\end{equation}
The monodromy matrix and the characteristic polynomial inherit symmetry from the underlying ODE. This can prove particularly advantageous when investigating the stability and instability of periodic traveling waves in Hamiltonian PDEs.

\begin{lemma}[Generalized Hamiltonian symmetry]\label{lem:symm}
Suppose that $\mathbf{A}(x,\lambda)$ is an $n\times n$ matrix-valued function, periodic in $x$, and {\rm (A1)} and {\rm (A2)} hold true. Let $\mathbf{M}(\lambda)$ denote the monodromy matrix of \eqref{eqn:A}, and it must satisfy
\begin{equation}\label{eqn:M&B}
{\bf M}(\lambda)={\bf B}^{-\top}(\lambda){\bf M}^{-\top}(-\lambda){\bf B}^\top(\lambda).
\end{equation}

Additionally, suppose that $\tr(\mathbf{A}(x,\lambda))=0$ for all $x\in\mathbb{R}$ and $\lambda\in\mathbb{C}$. Let the characteristic polynomial of the monodromy matrix take the form
\[
p(\mu,\lambda)=\sum_{k=0}^n(-\mu)^k e_{n-k}(\lambda),
\]
where $e_k(\lambda)$ denotes the $k$-th elementary symmetric polynomial of the eigenvalues of $\mathbf{M}(\lambda)$. That is,
\[
e_0(\lambda)=1\quad\text{and}\quad
e_k(\lambda)=\sum_{1\leq j_1 < j_2 <\ldots j_k \leq n}\mu_{j_1}\mu_{j_2}\cdots \mu_{j_k}, \quad k=1,2,\dots,n,
\]
where $\mu_{j_1}, \mu_{j_2}, \dots, \mu_{j_k}$  are the eigenvalues of $\mathbf{M}(\lambda)$. The elementary symmetric polynomials must then satisfy
\begin{align}
&e_k(\lambda)=e_{n-k}(-\lambda),\quad k=0,1,2,\dots,n,\quad \text{for $\lambda \in \mathbb{C}$}\label{eqn:e_k(C)}
\intertext{and, in turn,}
&e_k(\lambda)=\overline{e_{n-k}(\lambda)},\quad k=0,1,2,\dots,n,\quad \text{for $\lambda\in i\mathbb{R}$}.\label{eqn:e_k(iR)}
\end{align}
Particularly, for $n$ even, $e_{n/2}(\lambda)$ is necessarily real. 
\end{lemma}

\begin{proof}
Recalling \eqref{def:M}, we observe that $\mathbf{V}^{-\top}(x,\lambda)$ satisfies
\[
\mathbf{V}^{-\top}_x(x,\lambda)=-\mathbf{A}^\top(x,\lambda)\mathbf{V}^{-\top}(x,\lambda) \quad\text{and}\quad \mathbf{V}^{-\top}(0,\lambda)=\mathbf{I}_n,
\]
whence $\mathbf{W}(x,\lambda):=\mathbf{B}^{-1}(\lambda)\mathbf{V}^{-\top}(x,\lambda)\mathbf{B}(\lambda)$ satisfies
\[
\mathbf{W}_x(x,\lambda)=-\mathbf{B}^{-1}(\lambda)\mathbf{A}^\top(x,\lambda)\mathbf{B}(\lambda)\mathbf{W}(x,\lambda)=\mathbf{A}(x,-\lambda)\mathbf{W}(x,\lambda) \quad\text{and}\quad \mathbf{W}(0,\lambda)=\mathbf{I}_n,
\]
by \text{(A2)}. Additionally, recalling \eqref{def:M}, we observe that $\mathbf{V}(x,-\lambda)$ satisfies
\[
\mathbf{V}_x(x,-\lambda)=\mathbf{A}(x,-\lambda)\mathbf{V}(x,-\lambda) \quad\text{and}\quad \mathbf{V}(0,-\lambda)=\mathbf{I}_n. 
\]
Therefore, \eqref{eqn:M&B} follows by uniqueness. 

Suppose that $\tr(\mathbf{A}(x,\lambda))=0$ for all $x\in\mathbb{R}$ and $\lambda\in\mathbb{C}$, and we observe
\begin{align*}
\sum_{k=0}^n(-\mu)^ke_{n-k}(-\lambda)=&\det(\mathbf{M}(-\lambda)-\mu\mathbf{I}_n)
=\det(\mathbf{M}^{-1}(\lambda)-\mu\mathbf{I}_n)\\
=&(-\mu)^n\det(\mathbf{M}^{-1}(\lambda))\det(\mathbf{M}(\lambda)-\mu^{-1}\mathbf{I}_n)=\sum_{k=0}^n(-\mu)^ke_k(\lambda).
\end{align*}
Therefore, \eqref{eqn:e_k(C)} follows. Here the second equality follows from generalized Hamiltonian symmetry of the monodromy matrix, and the last equality follows because 
\[
\det(\mathbf{V}(x,\lambda))=\exp\left(\int^x_0\tr(\mathbf{A}(x,\lambda))~dx\right)\det(\mathbf{V}(0,\lambda))=\det(\mathbf{V}(0,\lambda)),
\]
by hypothesis. This completes the proof.
\end{proof}

Note that $e_k(\lambda)$ can be related to $\tr({\bf M}^j(\lambda))$, $j=1,2,\dots,k$, using the well-known Newton formula as
\begin{equation}\label{eqn:e&tr}
ke_k(\lambda)=\sum_{j=1}^k(-1)^{j-1}e_{k-j}(\lambda)\tr({\bf M}^j(\lambda)).
\end{equation}
We will primarily work with $e_k(\lambda)$, but \eqref{eqn:e&tr} allows us to express them in terms of $\tr(\mathbf{M}^k(\lambda))$, which can prove particularly advatageous for numerical computations. 

Lemma~\ref{lem:symm} implies---perhaps not surprisingly---that if \eqref{eqn:A} exhibits generalized Hamiltonian symmetry then only half of the invariants of $\mathbf{M}(\lambda)$, $\lambda \in i\mathbb{R}$, are linearly independent. However, the full implications of this observation have not been thoroughly explored. Our interest here lies in the spectrum along the imaginary axis, which holds significant importance when investigating the stability and instability of periodic traveling waves in Hamiltonian PDEs. Our objective is to determine the algebraic multiplicity of the $L^2(\mathbb{R})$ essential spectrum of \eqref{eqn:JH} on the imaginary axis, as well as the bifurcation of the spectrum away from the axis, by utilizing the {\em Floquet discriminant} $\boldsymbol{f}:i{\mathbb R} \to {\mathbb R}^{n-1}$, defined element-wise as
\[
f_k(\lambda)=\begin{cases}\Re(e_{(k+1)/2}(\lambda)),\quad & 1\leq k \leq n-1, \text{odd}, \\
\Im(e_{k/2}(\lambda)),\quad & 1\leq k \leq n-1, \text{even}.\end{cases}
\]
We identify $\mathbb{R}^2$ with $\mathbb{C}$ whenever convenient, and $\boldsymbol{f}:i{\mathbb R} \to \begin{cases}{\mathbb C}^{(n-1)/2}, \quad & n \text{~~odd}, \\
{\mathbb C}^{n/2-1}\times {\mathbb R}, \quad & n \text{~~even}.
\end{cases}$ 
For $n=2$, note from \eqref{eqn:e_k(iR)} and \eqref{eqn:e&tr} that the Floquet discriminant becomes $\tr(\mathbf{M}(\lambda))$, which is real for $\lambda\in i\mathbb{R}$. For $n=3$, $\tr(\mathbf{M}(\lambda))$ takes values in $\mathbb{C}$, and for $n=4$, the Floquet discriminant consists of $\tr(\mathbf{M}(\lambda))$ and $\tfrac12(\tr(\mathbf{M}(\lambda))^2-\tr(\mathbf{M}^2(\lambda)))$. The latter is necessarily real.

\subsection{Spectrum on the imaginary axis}\label{sec:imaginary axis}

If the monodromy matrix of \eqref{eqn:A}, evaluated at $\lambda \in \mathbb{C}$, possesses $m$ simple eigenvalues along the unit circle---that is, the characteristic polynomial of the monodromy matrix has $m$ simple roots on the unit circle---then it follows from the Floquet theorem that $\lambda$ belongs to the $L^2(\mathbb{R})$ essential spectrum of \eqref{eqn:JH} with an algebraic multiplicity $m$. Consequently, our interest lies in determining the number of eigenvalues of the monodromy matrix on the unit circle. 
This task can be intricate, but we will explore how generalized Hamiltonian symmetry can facilitate the process.

Suppose that $\mathbf{A}(x,\lambda)$ is an $n\times n$ matrix-valued function, periodic in $x$, satisfying {\rm (A1)} and {\rm (A2)}. We deduce from \eqref{eqn:e_k(iR)} that the characteristic polynomial of the monodromy matrix of \eqref{eqn:A} obeys
\[
p(\mu,\lambda)=\mu^n \overline{p\left(\tfrac{1}{\overline{\mu}},\lambda\right)} \quad \text{for $\lambda\in i\mathbb{R}$},
\]
whence if $\mu$ is a root of $p(\cdot,\lambda)$ for $\lambda\in i\mathbb{R}$ then $\tfrac{1}{\overline{\mu}}$ is also a root necessarily with the same multiplicity. In other words, the roots of $p(\cdot,\lambda)$ for $\lambda \in i\mathbb{R}$ are symmetric with respect to reflection across the unit circle. To determine the number of roots on the unit circle, we exploit the fact that linear fractional transformations map generalized circles to generalized circles. More specifically, we introduce
\begin{equation}\label{def:pp}
p^\sharp(\nu,\lambda):=(1- i\nu)^n p\left(\frac{1+i\nu}{1 -i\nu},\lambda\right),
\end{equation}
which is a real polynomial for $\lambda\in i\mathbb{R}$, and the real roots of $p^\sharp(\nu,\lambda)$ for $\lambda\in i\mathbb{R}$ correspond to the roots of $p(\mu,\lambda)$ on the unit circle. (Care must be taken, though, when accounting for roots at $\infty$.) Importantly, the discriminant transforms under the action of $\operatorname{SL}(2,{\mathbb C})$. Specifically, if $p$ is a polynomial of degree $n$ then
\[
\disc_\nu\left((c+d\nu)^n p\left(\frac{a+b\nu}{c+d\nu}\right)\right)=(ad-bc)^{n(n-1)}\disc_\mu p(\mu),
\]
where ``$\disc$'' means the discriminant. For this and other facts concerning the classical discriminant we refer the reader to the text of Gelfand, Kapranov and Zelevinsky\cite{GKZ}. Consequently, $\disc_\nu p^\sharp(\nu,\lambda)=(-2i)^{n(n-1)}\disc_\mu p(\mu,\lambda)$, and conditions involving the sign of the discriminant of $p(\mu,\lambda)$ can be expressed in terms of the sign of $\disc_\nu p^\sharp(\nu,\lambda)$. (Care must be taken, though, because the degree of $p^\sharp$ is less than the degree of $p$ when $p$ has a root of $-1$.) It is important to note that the discriminant of $p^\sharp(\nu,\lambda)$ and, hence, $\disc_\mu p(\mu,\lambda)$ are real for $\lambda \in i\mathbb{R}$. See, for instance, \cite[Chapter~12]{GKZ} for more on discriminant and resultant. 

Consequently, the task of counting the number of roots on the unit circle of the characteristic polynomial $p(\mu,\lambda)$ for $\lambda \in i\mathbb{R}$ becomes equivalent to counting the number of real roots of the real polynomial $p^\sharp(\nu,\lambda)$, defined in \eqref{def:pp}. Such counting can be accomplished, for instance, by utilizing the Sturm sequence of a polynomial. When dealing with real polynomials with a degree $\leq 3$, the number of real roots can be determined by inspecting the discriminant of the polynomial. For polynomials with a degree $\geq 4$, additional auxiliary quantities, which depend polynomially on the coefficients of the polynomial, must be taken into account, alongside the discriminant. 

Let's examine a somewhat elementary example, involving a second-order equation 
\begin{equation}\label{eqn:JH2}
v_{xx}+Q(x)v=\lambda^2v, 
\end{equation}
where $Q(x)$ is a real-valued and periodic function. We remark that the left side of \eqref{eqn:JH2} makes a self-adjoint operator. We can reformulate \eqref{eqn:JH2} as
\begin{equation}\label{eqn:A2}
\mathbf{v}_x=\begin{pmatrix}
0 & \lambda \\
\lambda - Q(x)/\lambda & 0 
\end{pmatrix}\mathbf{v}=:\mathbf{A}(x,\lambda)\mathbf{v},
\qquad \lambda \in\mathbb{C},
\end{equation}
and we confirm that (A1) and (A2) hold true for $\mathbf{B}=\begin{pmatrix}
0 & 1 \\ 1 & 0 \end{pmatrix}$. Incidentally, \eqref{eqn:A2} does not take the usual symplectic form of the spectral problem for the Schr\"odinger operator, for which $\mathbf{A}^\top(x,\lambda)\mathbf{B}=-\mathbf{B}\mathbf{A}(x,\lambda)$ for some skew-symmetric matrix $\mathbf{B}$. 

Let $\mathbf{M}(\lambda)$ denote the $2\times 2$ monodromy matrix of \eqref{eqn:A2}. The characteristic polynomial of the monodromy matrix takes the form
\[
p(\mu,\lambda)=\mu^2-\tr({\bf M}(\lambda))\mu+1,
\]
and \eqref{def:pp} becomes $p^\sharp(\nu,\lambda)=-(2+\tr(\mathbf{M}(\lambda)))\nu^2+2-\tr(\mathbf{M}(\lambda))$. Consequently,
\[
\disc_\nu p^\sharp(\nu,\lambda)= 4(4-\tr({\bf M}(\lambda))^2).  
\]
This reproduces the well-known result \cite{MagnusWinkler1966,Eastham1973,StarYak1975} that the $L^2(\mathbb{R})$ essential spectrum of \eqref{eqn:A2} has a band when the Floquet discriminant $\tr({\bf M}(\lambda))$, which is real for $\lambda \in i\mathbb{R}$, lies within the interval $(-2,2)$, a band edge with a double eigenvalue of $1$ or $-1$ when $\tr({\bf M}(\lambda))=\pm2$, and a gap otherwise.  

When dealing with higher-order equations that exhibit generalized Hamiltonian symmetry, we wish to likewise establish that the monodromy matrix of \eqref{eqn:A} at $\lambda \in i\mathbb{R}$ has $m$ simple eigenvalues on the unit circle, or, equivalently, $\lambda$ belongs to the $L^2(\mathbb{R})$ essential spectrum of \eqref{eqn:JH} with an algebraic multiplicity $m$, if and only if the Floquet discriminant lies within a region $\varOmega_m\subset \mathbb{R}^{n-1}$, which can be explicitly computed. For example, for $n=3$ and considering the generalized KdV equation, we can demonstrate that $\varOmega_3 \subset \mathbb{R}^2$ makes to a deltoidal region, bounded by a hypocycloid with three cusps, and $\lambda \in i\mathbb{R}$ is in the $L^2(\mathbb{R})$ essential spectrum with an algebraic multiplicity three if and only if the Floquet discriminant $\tr(\mathbf{M}(\lambda))$ lies within $\varOmega_3$. Similarly, for $n=4$ and the nonlinear Schr\"odinger equation, $\varOmega_4 \subset \mathbb{R}^3$ makes a tetrahedral region with four cusps. Theorems~\ref{thm:spec3} and \ref{thm:spec4} provide further details.

\subsection{Bifurcation of the spectrum away from the imaginary axis}\label{sec:bifurcation} 

We are also interested in Hopf bifurcations of the $L^2(\mathbb{R})$ essential spectrum of \eqref{eqn:JH} away from the imaginary axis, which can manifest not only in the vicinity of the origin in the complex plane---namely, modulational instability---but also away from $0\in\mathbb{C}$. This holds significant importance when investigating the stability and instability of periodic traveling waves in Hamiltonian PDEs. For second-order equations, the third author and collaborators \cite{JMMP1,JMMP2,JMMP3,MM} have offered a comprehensive theoretical explanation of such bifurcations. For higher-order equations, however, our understanding remains limited, and only a few results {\bf CITE!} are available for small amplitudes, relying on perturbative techniques. Detecting all the spectrum off the imaginary axis for equations of all orders could be an insurmountable task. Here our objective is to derive a {\em bifurcation index}, a polynomial in the Floquet discriminant and its derivative, which can be utilized to locate points along the imaginary axis where bifurcations of the spectrum away from the axis may occur.   

We continue assuming that $\mathbf{A}(x,\lambda)$ is an $n\times n$ matrix-valued function, periodic in $x$, and {\rm (A1)} and {\rm (A2)} hold true. Suppose that the characteristic polynomial of the monodromy matrix of \eqref{eqn:A} satisfies
\[
p(\mu_0,\lambda_0)=0\quad \text{for some $\mu_0\in\mathbb{C}$, $|\mu_0|=1$,}\quad\text{for some $\lambda_0 \in i\mathbb{R}$}.
\]
(More generally, the characteristic polynomial may vanish on subsets of the unit circle, each with co-dimension one.) 
It then follows from the implicit function theorem that we can solve $p(\mu,\lambda)=0$ uniquely for $\lambda$ as a function of $\mu$ in a neighborhood of $\mu_0$ and $\lambda_0$, as long as $p_\lambda(\mu_0,\lambda_0)\neq0$. Consequently, a necessary condition for the $L^2(\mathbb{R})$ essential spectrum of \eqref{eqn:JH} to bifurcate at $\lambda_0\in i\mathbb{R}$ away from the imaginary axis in a transversal manner---in addition to the spectrum on the axis itself---is
\[
p_\lambda(\mu_0,\lambda_0)=0.
\]
Since both $p(\mu,\lambda)$ and $p_\lambda(\mu,\lambda)$ are polynomials in $\mu$, they will simultaneously vanish if and only if their resultant with respect to $\mu$ vanishes. Therefore, a necessary condition for the bifurcation of the spectrum away from the imaginary axis is
\begin{equation}\label{eqn:res}
\res_\mu(p(\mu,\lambda),p_\lambda(\mu,\lambda))=0\quad\text{for some $\lambda\in i\mathbb{R}$},
\end{equation}
where ``$\res$" means the resultant.

For example, let's consider \eqref{eqn:A2} and a straightforward calculation reveals that
\[
\res_\mu(p(\mu,\lambda),p_\lambda(\mu,\lambda))=\tr({\bf M}_\lambda(\lambda))^2.
\]
Indeed, if the $L^2(\mathbb{R})$ essential spectrum of \eqref{eqn:A2} bifurcates at $\lambda \in i\mathbb{R}$ away from the imaginary axis then $\tr({\bf M}_\lambda(\lambda))=0$. Conversely, it is well-known \cite{MagnusWinkler1966} that the Floquet discriminant for a second-order self-joint operator is monotonous within bands. 

For third-order equations, under some additional assumptions, we can establish that \eqref{eqn:res} is also a sufficient condition for the bifurcation of the spectrum away from the imaginary axis. Theorem~\ref{thm:bifur3} provides further details.

\subsection{Main results}\label{sec:result}

For second-order self-adjoint operators, such as \eqref{eqn:A2}, the Floquet discriminant, the trace of the monodromy matrix, can be used to determine whether the $L^2(\mathbb{R})$ essential spectrum has a band or a gap. For third-order and fourth-order equations that exhibit generalized Hamiltonian symmetry, the Floquet discriminant can likewise determine the algebraic multiplicity of the spectrum along the imaginary axis. 

\begin{theorem}[Spectrum on the imaginary axis for third-order equations]\label{thm:spec3}
Suppose that $\mathbf{A}(x,\lambda)$ is a $3\times 3$ matrix-valued function, periodic in $x$, satisfying {\rm (A1)} and {\rm (A2)}. Suppose that $\tr(\mathbf{A}(x,\lambda))=0$ for all $x\in\mathbb{R}$ and $\lambda\in\mathbb{C}$. Let $\mathbf{M}(\lambda)$ denote the monodromy matrix of \eqref{eqn:A}, and $\mathbf{M}(\lambda)$ for $\lambda \in i\mathbb{R}$ must have either one or three eigenvalues along the unit circle, counted by algebraic multiplicity.

Recall that the Floquet discriminant is $\boldsymbol{f}(\lambda)=\tr(\mathbf{M}(\lambda))$ for $\lambda \in i\mathbb{R}$, and we write $\boldsymbol{f}(\lambda)=f_1(\lambda)+if_2(\lambda)$. Let
\begin{equation}\label{def:disc3}
\begin{aligned}
\Delta_3(\lambda)=&|\boldsymbol{f}(\lambda)|^4-4\boldsymbol{f}(\lambda)^3-4\overline{\boldsymbol{f}(\lambda)}^3+18|\boldsymbol{f}(\lambda)|^2-27 \\ 
=&((f_1^2+f_2^2)^2-8(f_1^3-3f_1f_2^2)+18(f_1^2+f_2^2)-27)(\lambda),
\end{aligned}
\end{equation}
and the followings hold true:
\begin{itemize}
\item If $\Delta_3(\lambda)<0$ then $\mathbf{M}(\lambda)$ has three distinct eigenvalues on the unit circle, implying that $\lambda \in i\mathbb{R}$ belongs to the $L^2(\mathbb{R})$ essential spectrum of \eqref{eqn:JH} with an algebraic multiplicity three.
\item If $\Delta_3(\lambda)>0$, on the other hand, then $\mathbf{M}(\lambda)$ has three eigenvalues: one on the unit circle, one inside the unit circle, and one outside the unit circle. The latter two have the same argument.
\item If $\Delta_3(\lambda)=0$ then $\mathbf{M}(\lambda)$ has three eigenvalues on the unit circle, counted with their algebraic multiplicities, and at least two of them are degenerate. 
\end{itemize}

Alternatively, let 
\begin{equation}\label{def:deltoid}
\varGamma=\{2e^{i\theta}+e^{-2i\theta} \in \mathbb{C}:\theta \in [-\pi,\pi]\}
\end{equation}
denote the deltoid curve or a Steiner curve, a hypocycloid with three cusps, and the followings hold true:
\begin{itemize}
\item If $\boldsymbol{f}(\lambda)$ lies inside $\varGamma$, the connected component of $\mathbb{C}\setminus\varGamma$ containing $0$, then $\mathbf{M}(\lambda)$ has three distinct eigenvalues on the unit circle.
\item If $\boldsymbol{f}(\lambda)$ lies outside $\varGamma$, the connected component of $\mathbb{C}\setminus\varGamma$ containing $\infty$, then $\mathbf{M}(\lambda)$ has three eigenvalues: one on the unit circle, one inside the unit circle, and one outside the unit circle.
\item If $\boldsymbol{f}(\lambda)$ lies on $\varGamma$ then $\mathbf{M}(\lambda)$ has three eigenvalues on the unit circle and at least one eigenvalue has a higher multiplicity. Specifically, one eigenvalue has an algebraic multiplicity one and another has a multiplicity two unless $\boldsymbol{f}(\lambda)$ coincides with one of the cusps of $\varGamma$---$3$, $3e^{2\pi i/3}$, and $3e^{4\pi i/3}$---in which cases $1$, $e^{2\pi i/3}$, and $e^{4\pi i/3}$ respectively are eigenvalues with an algebraic multiplicity three.  
\end{itemize}
\end{theorem}

\begin{proof}
The proof begins by observing that the characteristic polynomial of the monodromy matrix of \eqref{eqn:A} takes the form
\begin{align}
p(\mu,\lambda)=&-\mu^3+\boldsymbol{f}(\lambda)\mu^2-\boldsymbol{f}(-\lambda)\mu+1 , \qquad \lambda\in\mathbb{C}, \notag 
\intertext{where $\boldsymbol{f}(\lambda)=\tr(\mathbf{M}(\lambda))$, and}
p(\mu,\lambda)=&-\mu^3+\boldsymbol{f}(\lambda)\mu^2-\overline{\boldsymbol{f}(\lambda)}\mu+1, \qquad \lambda \in i\mathbb{R}, \label{def:p3}
\end{align}
by \eqref{eqn:e_k(C)} and \eqref{eqn:e_k(iR)}. Consequently, if $\mu$ is a root of $p(\cdot,\lambda)$ for $\lambda\in i\mathbb{R}$ then $\tfrac{1}{\overline{\mu}}$ is also a root. Since $p(\mu,\lambda)$, $\lambda\in i\mathbb{R}$, has three roots, counted with their algebraic multiplicities, at least one root must satisfy $\mu=\tfrac{1}{\overline{\mu}}$, implying that it lies on the unit circle. Consequently, $p(\mu,\lambda)$, $\lambda\in i\mathbb{R}$, has either one or three roots on the unit circle. 

To determine whether all three roots of $p(\mu,\lambda)$, $\lambda \in i\mathbb{R}$, lie on the unit circle or if two of them fall off it, we can inspect the sign of the discriminant. Specifically, the number of roots changes if and only if $p(\mu,\lambda)$, $\lambda \in i\mathbb{R}$, has a double root or, equivalently, the discriminant becomes zero. Let $\boldsymbol{f}=f_1+if_2$, and $\disc(-\mu^3+(f_1+if_2)\mu^2-(f_1-if_2)\mu+1)$
gives rise to $\Delta_3$, defined in \eqref{def:disc3}. 

Additionally, a direct calculation reveals that the deltoid curve, defined by \eqref{def:deltoid}, in the complex coordinates, or, equivalently, 
\[
(2\cos(\theta)+\cos(2\theta),2\sin(\theta)-\sin(2\theta))
\]
in polar coordinates, corresponds to $\Delta_3=0$ in the $(f_1,f_2)$ coordinates. This completes the proof.
\end{proof}

\begin{figure}[htbp]
\centering
\includegraphics[scale=0.55]{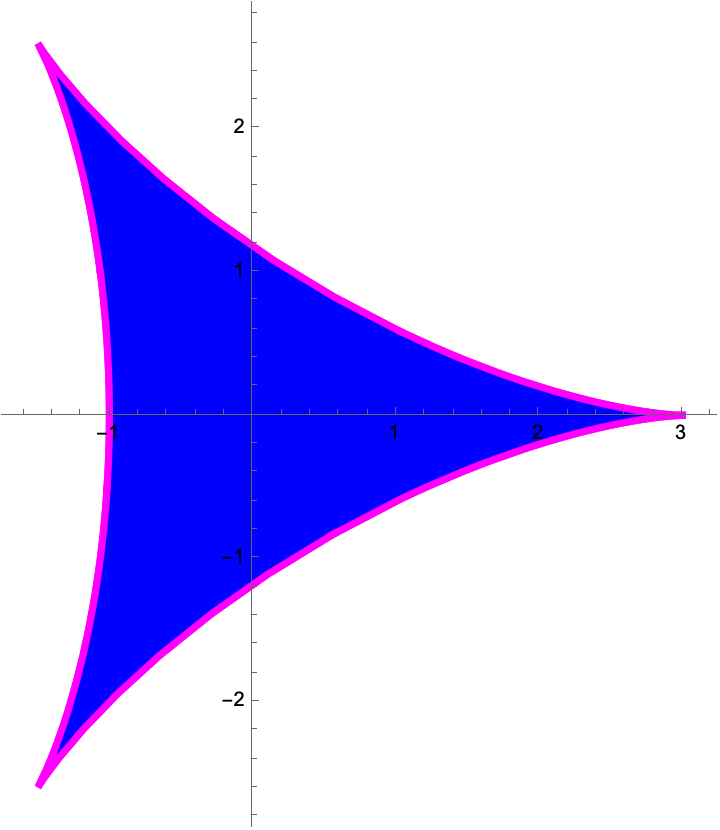}
\includegraphics[scale=0.55]{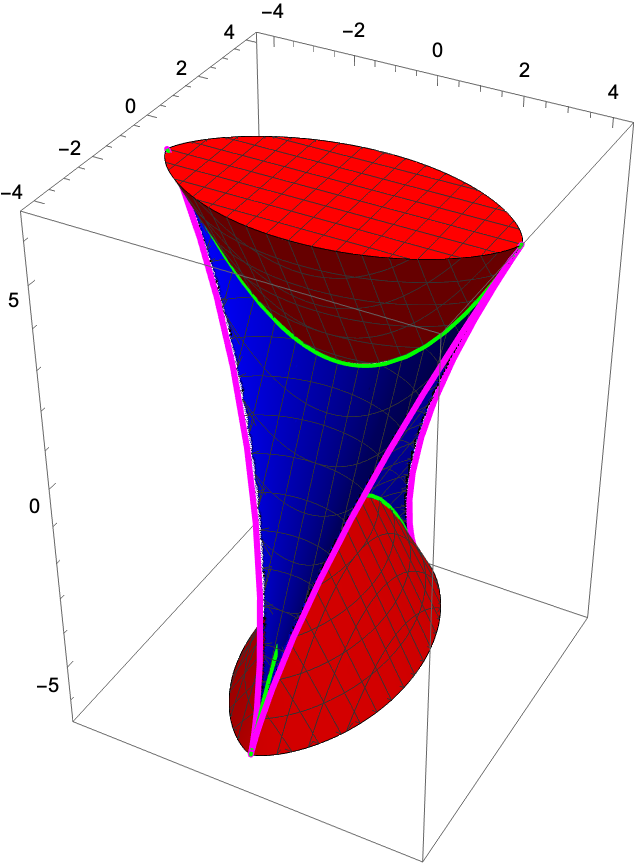}
\caption{The deltoidal region in $\mathbb{R}^2$ (left) and tetrahedral region in $\mathbb{R}^3$ (right). For third-order and fourth-order equations with generalized Hamiltonian symmetry, the monodromy matrix has three and four simple eigenvalues, respectively, along the unit circle---or, equivalently, the $L^2(\mathbb{R})$ essential spectrum of \eqref{eqn:JH} has an algebraic multiplicity three and four, respectively, on the imaginary axis---if and only if the Floquet discriminant lies in the blue region.}
\label{fig:region34}
\end{figure}
 
The left panel of Figure~\ref{fig:region34} depicts the region in $\mathbb{R}^2(=\mathbb{C}$), enclosed by the deltoid curve $\varGamma$, defined in \eqref{def:deltoid}. The monodromy matrix of \eqref{eqn:A} at $\lambda \in i\mathbb{R}$ possesses three simple eigenvalues on the unit circle, or, equivalently, $\lambda$ belongs to the $L^2(\mathbb{R})$ essential spectrum of \eqref{eqn:JH} with an algebraic multiplicity three, if and only if the Floquet discriminant lies within the deltoidal region.

\begin{remark*}
The deltoid curve $\varGamma$, defined in \eqref{def:deltoid} and also known as a Steiner curve, is a hypocycloid with three cusps, and it is the path traced by a point on the circumference of a circle with a radius $1$ as it rolls without slipping along the interior of a circle with a radius $3$. The fact that $\varGamma$ shows up in Theorem~\ref{thm:spec3} may not be as surprising as it appears. In fact, the range of the trace operator
\[
\tr: \operatorname{SU}(3) \to \mathbb{C}
\]
coincides with $\varGamma$ and its interior \cite{kaiser}. Consequently, if $\tr(\mathbf{M}(\lambda))$ lies within the interior of $\varGamma$ then $\mathbf{M}(\lambda)$ is similar to a matrix in $\operatorname{SU}(3)$. This suggests a connection among the deltoid curve, the range of the trace operator, and the monodromy matrix.
\end{remark*}
 
\begin{theorem}[Spectrum on the imaginary axis for fourth-order equations]\label{thm:spec4}
Suppose that $\mathbf{A}(x,\lambda)$ is a $4\times 4$ matrix-valued function, periodic in $x$, satisfying {\rm (A1)} and {\rm (A2)}. Suppose that $\tr(\mathbf{A}(x,\lambda))=0$ for all $x\in\mathbb{R}$ and $\lambda\in\mathbb{C}$. Let $\mathbf{M}(\lambda)$ denote the monodromy matrix of \eqref{eqn:A}, and for $\lambda \in i\mathbb{R}$, it must have zero, two, or four eigenvalues along the unit circle, counted by algebraic multiplicity. 

Recall that the Floquet discriminant is $\boldsymbol{f}(\lambda)=(f_1(\lambda),f_2(\lambda),f_3(\lambda))\in \mathbb{R}^3 (=\mathbb{C}\times \mathbb{R})$ for $\lambda\in i\mathbb{R}$, where
\begin{equation}\label{def:f123}
f_1(\lambda)+if_2(\lambda)=\tr(\mathbf{M}(\lambda))\quad\text{and}\quad
f_3(\lambda)=\tfrac12(\tr(\mathbf{M}(\lambda))^2-\tr(\mathbf{M}^2(\lambda))).
\end{equation}
Let
\begin{equation}\label{def:disc4}
\begin{aligned}
\Delta_4=&-4(f_1^6+f_2^6)-12f_1^2f_2^2(f_1^2+f_2^2)+(f_1^2+f_2^2)^2f_3^2\\
&+36(f_1^4-f_2^4)f_3-8(f_1^2-f_2^2)f_3^3\\
&-60(f_1^4+f_2^4)+312f_1^2f_2^2+16f_3^4-80(f_1^2+f_2^2)f_3^2 \\
&+288(f_1^2-f_2^2)f_3-192(f_1^2+f_2^2)-128f_3^2+256
\end{aligned}
\end{equation}
and 
\begin{align}
&P_4=8(2f_1+f_3+2)(2f_3-12)-48f_2^2, \label{def:P4}\\
&\begin{aligned}
D_4=-256(&4f_1^4+3f_2^4+f_1^2(4f_2^2+(-6+f_3)^2)\\
&+f_2^2(28+12f_3-f_3^2)+4f_1^3(2+f_3)\\
&-4(-2+f_3)(2+f_3)^2+16f_1(4+2f_2^2-f_3^2)).\end{aligned} \label{def:D4}
\end{align}
Suppose that $\Delta_4(\lambda),P_4(\lambda),D_4(\lambda)\neq0$, and $\lambda \in i\mathbb{R}$ must belong to the $L^2(\mathbb{R})$ essential spectrum of \eqref{eqn:JH} with:
\begin{itemize}
\item multiplicity $4$ if $\Delta_4(\lambda)>0$, $P_4(\lambda),D_4(\lambda)<0$, 
\item multiplicity $2$ if $\Delta_4(\lambda)<0$, and 
\item multiplicity $0$ if $\Delta_4(\lambda)>0$ while $P_4(\lambda)$ or $D_4(\lambda)>0$.
\end{itemize}
\end{theorem}

\begin{proof}
The proof resembles that of Theorem~\ref{thm:spec3}, observing that the characteristic polynomial of the monodromy matrix of \eqref{eqn:A} takes the form
\begin{equation}\label{def:p4}
p(\mu,\lambda)=\mu^4-(f_1(\lambda)+if_2(\lambda))\mu^3+f_3(\lambda)\mu^2-(f_1(\lambda)-if_2(\lambda))\mu+1
\end{equation}
for $\lambda \in i\mathbb{R}$, by generalized Hamiltonian symmetry. Here $f_1$, $f_2$, $f_3$ are defined in \eqref{def:f123}. We remark that $f_3$ is necessarily real although $f_1+if_2$ does not have to be. However, $f_1+if_2$ becomes real, for instance, for the linearization of the nonlinear Schr\"odinger equation about a trivial phase solution, thanks to additional symmetry. 

The discriminant of \eqref{def:p4} leads to $\Delta_4$, defined in \eqref{def:disc4}. 
When $\Delta_4\neq 0$, the eigenvalues of the monodromy matrix of \eqref{eqn:A} on the unit circle are distinct, and these eigenvalues determine the algebraic multiplicity of the $L^2(\mathbb{R})$ essential spectrum of \eqref{eqn:JH} on the imaginary axis. To count the number of eigenvalues of the monodromy matrix on the unit circle, we can inspect the signs of \eqref{def:disc4}, \eqref{def:P4}, \eqref{def:D4}, provided that $\Delta_4$, $P_4$, $D_4\neq0$. We omit the details.
\end{proof}

The right panel of Figure~\ref{fig:region34} shows the regions in $\mathbb{R}^3(=\mathbb{C}\times \mathbb{R})$ that correspond to different algebraic multiplicities of the $L^2(\mathbb{R})$ essential spectrum of \eqref{eqn:JH} along the imaginary axis. 
The multiplicity of $\lambda\in i\mathbb{R}$ is four when the Floquet discriminant lies within the tetrahedral region, depicted in blue. The multiplicity is zero---that is, $\lambda$ is not in the spectrum---when the Floquet discriminant is within the red region, and the multiplicity is two in the remaining region of $\mathbb{R}^3$. 

The tetrahedral region is bounded in $\mathbb{R}^3$ because if all eigenvalues of the monodromy matrix lie on the unit circle then
\[
|(f_1,f_2)|\leq 4\quad\text{and}\quad |f_3|\leq 6.
\]
The region has four cusps located at $(f_1,f_2,f_3)=(\pm 4,0,6)$ and $(0,\pm 4,6)$, at which the characteristic polynomial of the monodromy matrix simplifies to $(\mu\mp1)^4$ and $(\mu\mp i)^4$ respectively. The blue and the red regions are tangent to each other along the parabolic segments, in green, which can be parametrized as
\[
(f_1,f_2,f_3)=(-4\cos(\theta),0,2+4\cos^2(\theta))\quad\text{and}\quad 
(0,-4\cos(\theta),-2-4\cos^2(\theta)),
\]
where $\theta\in[-\pi,\pi]$. Along these curves, the monodromy matrix of \eqref{eqn:A} has two distinct eigenvalues each with an algebraic multiplicity two. The remaining four edges of the tetrahedral region, in magenta, can be represented by the curve parametrized as
\[
(3\cos(\theta)+\cos(3\theta), -3\sin(\theta)+\sin(3\theta), 6\cos(2\theta)).
\]
Along these curves, the monodromy matrix has one eigenvalue with a multiplicity four. 

Interestingly, when the tetrahedral region is projected onto the $f_3=0$ plane, it forms a region bounded by an astroid curve in $\mathbb{R}^2$, a hypocycloid with four cusps, which can be parametrized as
\[
(f_1,f_2)=(3\cos(\theta)+\cos(3\theta), -3\sin(\theta)+\sin(3\theta)).
\]
Recall the connection between the range of $\tr: \operatorname{SU}(3) \to \mathbb{C}$ and the deltoid curve defined in \eqref{def:deltoid}. Similarly, the range of the trace operator $\tr: \operatorname{SU}(4)\to \mathbb{C}$ coincides with the astroid curve and its interior.

\begin{remark*}\rm
The deltoidal and tetrahedral regions in $\mathbb{R}^2$ and $\mathbb{R}^3$ can be seen as analogous to taking the interval $(-2,2)\subset \mathbb{R}$ for second-order self-adjoint operators and extending it to third-order and fourth-order equations with generalized Hamiltonian symmetry. Going further, it is possible to explicitly calculate a region in $\mathbb{R}^{n-1}$ for $n$-th order equations, such that the monodromy matrix of \eqref{eqn:A} at $\lambda \in i\mathbb{R}$ has $n$ simple eigenvalues on the unit circle or, equivalently, $\lambda$ is in the $L^2(\mathbb{R})$ essential spectrum of \eqref{eqn:JH} with an algebraic multiplicity $n$ if and only if the Floquet discriminant lies within the region. In Section~\ref{sec:5} we present an example for $n=5$. However, it is important to note that as the order of the equation increases, the Floquet discriminant becomes more challenging to compute analytically as well as numerically.

The region where the purely imaginary spectrum of an $n$-th order equation has an algebraic multiplicity $n$ is bounded in $\mathbb{R}^{n-1}$ with $n$ cusps, at which the characteristic polynomial of the monodromy matrix becomes $(\mu-e^{2\pi ik/n})^n$, $k=0,1,\dots, n-1$. On the other hand, the regions corresponding to multiplicities $<n$ are unbounded in $\mathbb{R}^{n-1}$.
\end{remark*}

For second-order self-adjoint operators, such as \eqref{eqn:A2}, the derivative of the Floquet discriminant determines whether the $L^2(\mathbb{R})$ essential spectrum bifurcates away from the imaginary axis. For third-order equations that exhibit generalized Hamiltonian symmetry, the Floquet discriminant and its derivative can likewise be used for a necessary condition for such bifurcations. Furthermore, under some additional assumptions, the necessary condition can also serve as a sufficient condition.

\begin{theorem}[Spectrum away from the imaginary axis for third-order equations]\label{thm:bifur3}
Suppose that $\mathbf{A}(x,\lambda)$ is a $3\times 3$ matrix-valued function, periodic in $x$, satisfying {\rm (A1)} and {\rm (A2)}. Suppose that $\tr(\mathbf{A}(x,\lambda))=0$ for all $x\in\mathbb{R}$ and $\lambda\in\mathbb{C}$. Let $\boldsymbol{f}(\lambda)$, $\lambda \in i\mathbb{R}$, denote the Floquet discriminant. If the $L^2(\mathbb{R})$ essential spectrum of \eqref{eqn:JH} bifurcates at $\lambda\in i\mathbb{R}$ away from the imaginary axis in a transversal manner, alongside the spectrum on the axis, then  
\begin{equation}\label{def:res3}
\varPhi_3(\lambda):=\boldsymbol{f}'(\lambda)^3 +\boldsymbol{f}'(-\lambda)^3+\boldsymbol{f}(\lambda)\boldsymbol{f}'(-\lambda)^2\boldsymbol{f}'(\lambda)+\boldsymbol{f}(-\lambda) \boldsymbol{f}'(\lambda)^2\boldsymbol{f}'(-\lambda)=0.
\end{equation}
Here and elsewhere, the prime means ordinary differentiation. Conversely, if $\varPhi_3(\lambda)=0$ and if 
\begin{equation}\label{cond:suff3}
\Delta_3(\lambda)\neq 0\quad\text{and}\quad 
\boldsymbol{f}'(-\lambda)\boldsymbol{f}''(\lambda) \neq -\boldsymbol{f}'(\lambda)\boldsymbol{f}''(-\lambda), 
\end{equation}
where $\Delta_3$ is in \eqref{def:disc3}, then the spectrum of \eqref{eqn:JH} bifurcates at $\lambda$ away from the imaginary axis. Moreover, the spectrum exhibits the normal form 
\[
\Delta \lambda = \alpha \sqrt{\Delta\mu}+o(\sqrt{\Delta \mu}) 
\quad \text{for $|\Delta \lambda|, |\Delta\mu|\ll1$}.
\]
\end{theorem}

Note that $\varPhi_3(\lambda)$ is real for $\lambda \in i\mathbb{R}$ by \eqref{eqn:e_k(iR)}. 
 
\begin{proof}
Recall from the proof of Theorem~\ref{thm:spec3} that the characteristic polynomial of the monodromy matrix of \eqref{eqn:A} is 
\[
p(\mu,\lambda)=-\mu^3+\boldsymbol{f}(\lambda)\mu^2-\boldsymbol{f}(-\lambda)\mu+1.
\] 
Suppose that $p(\mu_0,\lambda_0)=0$ for some $\mu_0\in\mathbb{C}$, $|\mu_0|=1$, for some $\lambda_0 \in i\mathbb{R}$. It follows from the implicit function theorem that we can solve $p(\mu,\lambda)=0$ uniquely for $\lambda$ as a function of $\mu$ in a neighborhood of $\mu_0$ and $\lambda_0$, provided that
\[
\boldsymbol{f}'(\lambda_0)\mu_0^2+\boldsymbol{f}'(-\lambda_0)\mu_0=p_\lambda(\mu_0,\lambda_0) \neq 0,
\]
whence $\boldsymbol{f}'(\lambda_0)\mu_0 +\boldsymbol{f}'(-\lambda_0) \neq 0$ because $p(0,\lambda)=1$ for all $\lambda \in \mathbb{C}$. Consequently, if the implicit function theorem fails for $p(\mu,\lambda)=0$ near $\mu_0$ and $\lambda_0$ then
\begin{equation}\label{eqn:mu0}
\boldsymbol{f}'(\lambda_0)\mu_0+\boldsymbol{f}'(-\lambda_0)=0,
\end{equation}
implying
\[
\mu_0=-\frac{\boldsymbol{f}'(-\lambda_0)}{\boldsymbol{f}'(\lambda_0)}\quad \text{or}\quad \boldsymbol{f}'(\lambda_0)=0. 
\]  
Note that $-\frac{\boldsymbol{f}'(-\lambda_0)}{\boldsymbol{f}'(\lambda_0)}$ for $\lambda_0\in i\mathbb{R}$ necessarily lies on the unit circle although it does not have to be the case for $\lambda_0 \notin i\mathbb{R}$. Moreover, if $\boldsymbol{f}'(\lambda_0)=0$ for $\lambda_0 \in i\mathbb{R}$ then $\boldsymbol{f}'(-\lambda_0)=0$. Therefore, if the implicit function theorem fails for $p(\mu,\lambda)=0$ near $\mu_0$ and $\lambda_0$ then 
\begin{multline*}
p\left(-\frac{\boldsymbol{f}'(-\lambda_0)}{\boldsymbol{f}'(\lambda_0)},\lambda_0\right) \\ 
=\frac{1}{\boldsymbol{f}'(\lambda_0)^3}(\boldsymbol{f}'(-\lambda_0)^3+\boldsymbol{f}(\lambda_0)\boldsymbol{f}'(-\lambda_0)^2 \boldsymbol{f}'(\lambda_0)+\boldsymbol{f}(-\lambda_0)\boldsymbol{f}'(-\lambda_0)\boldsymbol{f}'(\lambda_0)^2+\boldsymbol{f}'(\lambda_0)^3)=0,
\end{multline*}
which gives rise to $\varPhi_3$, defined in \eqref{def:res3}. Alternatively, we can calculate 
\[
\res_\mu(p(\mu,\lambda), p_\lambda(\mu,\lambda))=\boldsymbol{f}'(\lambda)^3 +\boldsymbol{f}'(-\lambda)^3+\boldsymbol{f}(\lambda)\boldsymbol{f}'(-\lambda)^2\boldsymbol{f}'(\lambda)+\boldsymbol{f}(-\lambda) \boldsymbol{f}'(\lambda)^2\boldsymbol{f}'(-\lambda).
\]
Recall that the resultant of two polynomials is zero if and only if they share a common root. Therefore, if the implicit function theorem fails for $p(\mu,\lambda)=0$ near $\mu_0$ and $\lambda_0$ then  $\res_\mu(p(\mu_0,\lambda_0), p_\lambda(\mu_0,\lambda_0))=0$.

Conversely, suppose that $\varPhi_3(\lambda_0)=0$ for some $\lambda_0 \in i\mathbb{R}$, implying that
\[
p(\mu_0,\lambda_0)=p_\lambda(\mu_0,\lambda_0)=0 \quad \text{for some $\mu_0\in\mathbb{C}$, $|\mu_0|=1$, satisfying \eqref{eqn:mu0}}.
\]
Suppose that \eqref{cond:suff3} holds true. The former holds true if and only if $p(\mu,\lambda_0)=0$ has three distinct roots, whence $p_\mu(\mu_0,\lambda_0)\neq 0$. The latter, on the other hand, can be written as
\[
\mu_0\boldsymbol{f}''(\lambda_0)\neq \boldsymbol{f}''(-\lambda_0),
\]
by \eqref{eqn:mu0}, whence $p_{\lambda\lambda}(\mu_0,\lambda_0)\neq0$. We wish to solve 
\begin{align*}
0&=p(\mu,\lambda)\\&=p(\mu_0,\lambda_0)+p_\lambda(\mu_0,\lambda_0)\Delta \lambda+p_\mu(\mu_0,\lambda_0)\Delta \mu+\tfrac12 p_{\lambda\lambda}(\mu_0,\lambda_0)(\Delta \lambda)^2+o((\Delta\lambda)^2, \Delta\mu) \\
&=\tfrac12 p_{\lambda\lambda}(\mu_0,\lambda_0)(\Delta \lambda)^2+p_\mu(\mu_0,\lambda_0)\Delta \mu+o((\Delta\lambda)^2, \Delta\mu),
\end{align*}
where we assume that $|\Delta \mu|$ and $|\Delta \lambda|$ are sufficiently small. It follows from the Weierstrass preparation theorem that 
\[
\Delta\lambda=\pm i\sqrt{\frac{2p_\mu(\mu_0,\lambda_0)}{p_{\lambda\lambda}(\mu_0,\lambda_0)}\Delta \mu}+o(\sqrt{|\Delta \mu|})  
\]
for $|\Delta \mu|\ll1$. Since $\mu$ must lie on the unit circle, we can take $\Delta\mu=i\mu_0r$, where $r\in\mathbb{R}$. Since $p_{\lambda\lambda}(\mu_0,\lambda_0)\neq0$ and since there must be a spectral curve along the imaginary axis in the vicinity of $\mu_0$ and $\lambda_0$, $\frac{2p_\mu(\mu_0,\lambda_0)}{p_{\lambda\lambda}(\mu_0,\lambda_0)}i\mu_0 r$ must be real. As $r$ varies over positive real values, $\Delta\lambda$ varies over imaginary values, and as $r$ varies over negative real values, $\Delta\lambda$ varis over real values, or vice versa, depending on the sign of $\frac{p_\mu(\mu_0,\lambda_0)}{p_{\lambda\lambda}(\mu_0,\lambda_0)}i\mu_0$. In either case, two spectral curves emerge near $\mu_0$ and $\lambda_0$: one along the imaginary axis and another parallel to the real axis. This completes the proof. 
\end{proof}

\begin{remark*}
Let $\mu = e^{i\theta}$, where $\theta\in\mathbb{R}$. Treating $\theta$ as a function of $\lambda \in i\mathbb{R}$, we calculate
\[
\frac{d\theta}{d\lambda}=-\frac{1}{i\mu}\frac{p_\lambda(\mu,\lambda)}{p_\mu(\mu,\lambda)}=-\frac{\boldsymbol{f}'(\lambda)e^{2i\theta}+\boldsymbol{f}'(-\lambda)e^{i\theta}}{ie^{i\theta}(-3e^{2i\theta}+2\boldsymbol{f}(\lambda)e^{i \theta}-\boldsymbol{f}(-\lambda))}.
\]
We observe that $p_\lambda(\mu,\lambda)=\boldsymbol{f}'(\lambda)\mu^2+\boldsymbol{f}'(-\lambda)\mu$ vanishes if and only if $\varPhi_3(\lambda)=0$. Similarly, $i\mu p_\mu(\mu,\lambda)=-3i\mu^3+2i\boldsymbol{f}(\lambda)\mu^2-i\boldsymbol{f}(-\lambda)\mu$ vanishes if and only if $\Delta_3(\lambda)=0$. Consequently, the bifurcation of the $L^2(\mathbb{R})$ essential spectrum of \eqref{eqn:JH} away from the imaginary axis becomes possible when one of the eigenvalues of the monodromy matrix reverses its direction and starts moving the unit circle in the opposite direction as defined by the Krein signature \cite{Kapitula,KollarMiller}.
\end{remark*}

For higher-order equations with generalized Hamiltonian symmetry, a necessary condition for the $L^2(\mathbb{R})$ essential spectrum of \eqref{eqn:JH} to bifurcate away from the imaginary axis is \eqref{eqn:res}. 

\begin{theorem}[Spectrum away from the imaginary axis for higher-order equations]\label{thm:bifur+}
Suppose that $\mathbf{A}(x,\lambda)$ is a $n\times n$ matrix-valued function, periodic in $x$, satisfying {\rm (A1)} and {\rm (A2)}. Let $p(\mu,\lambda)$ denote the characteristic polynomial of the monodromy matrix of \eqref{eqn:A}. 
If the $L^2(\mathbb{R})$ essential spectrum of \eqref{eqn:JH} bifurcates at $\lambda\in i\mathbb{R}$ away from the imaginary axis in a transversal manner, alongside the spectrum on the axis, then    
\begin{equation}\label{def:res+}
\varPhi_n(\lambda):=\res_\mu(p(\mu,\lambda),p_\lambda(\mu,\lambda))=0.
\end{equation}
\end{theorem}

\begin{remark*}
Under some assumptions, \eqref{def:res+} implies that the mapping $\mu \mapsto \lambda(\mu)$ is locally not invertible. However, unlike the case with third-order equations, this is no longer guaranteed for an eigenvalue of the monodromy matrix on the unit circle. We present numerical examples that illustrate the scenario in which the bifurcation index becomes zero, yet no bifurcation of the spectrum takes place away from the imaginary axis.
\end{remark*}

If the eigenfunctions of \eqref{eqn:JH} exhibit an asymptotic behavior consistent with their WKB approximations at leading order as $\lambda$ varies along the imaginary axis towards $\pm i\infty$, and if these approximations coincide with the eigenfunctions of the limiting eigenvalue problem, with constant coefficients, then it becomes possible to determine the algebraic multiplicity of the eigenvalues $\lambda \in i\mathbb{R}$ for $|\lambda|\gg1$. Additioanlly, Theorem~\ref{thm:bifur+} can be applied to establish that, for some classes of quasi-periodic eigenvalue problems for Hamiltonian PDEs, the $L^2(\mathbb{R})$ essential spectrum remains along the imaginary axis outside a bounded region of the complex plane. This proves advantageous for numerical computations. It is worth noting that the first two authors and their collaborator \cite{BHS2023} present an alternative method for bounding unstable spectra $\notin i\mathbb{R}$, by utilizing an argument based on the Gershgorin circle theorem. 

\begin{theorem}[Spectrum away from the imaginary axis towards $\pm i\infty$]\label{thm:asym}
Consider
\begin{equation}\label{eqn:a}
\lambda v=(\partial_x^{n}+a_1(x)\partial_x^{n-1}+\cdots+a_n(x))v,\qquad \lambda\in\mathbb{C},
\end{equation}
where $a_k(x)$, $k=1,2,\dots, n$, are real-valued and smooth functions, satisfying $a_k(x+T)=a_k(x)$ for some $T>0$, the period. For $n$ odd, there exist only a finite number of zeros of $\varPhi_n(\lambda)$ for $\lambda\in i\mathbb{R}$, defined in \eqref{def:res+}, where $p(\mu,\lambda)$ denotes the characteristic polynomial of the monodromy matrix associated with \eqref{eqn:a}. Consequently, there are only a finite number of points along the imaginary axis where the $L^2(\mathbb{R})$ essential spectrum of \eqref{eqn:a} bifurcates away from the axis in a transversal manner.
\end{theorem}

\begin{proof}
Utilizing a WKB approximation to \eqref{eqn:a} (see, for instance, \cite[Chapter~7]{Was18} for details), we establish that the fundamental solutions of \eqref{eqn:a} satisfy 
\[
v_k(x,\lambda)\sim e^{\lambda^{1/n}\omega_k x}, \qquad k=1,2,\dots,n,
\]
for $\lambda\in i\mathbb{R}$ as $|\lambda|\to\infty$, where $\omega_k^n = 1$, that is, $\omega_k$ represent the $n$-th roots of unity. Introducing $\theta_k(\lambda)=e^{\lambda^{1/n} \omega_k T}$, $k=1,2,\dots,n$, we can approximate the characteristic polynomial of the monodromy matrix associated with \eqref{eqn:a} as
\[
p(\mu,\lambda) \sim \prod_{k=1}^n (\theta_k(\lambda)-\mu) =  \sum_{k=0}^n(-1)^{n-k}e_k(\theta_1,\theta_2,\dots, \theta_n)(\lambda) \mu^{n-k} 
\]
for $\lambda\in i\mathbb{R}$ as $|\lambda|\to\infty$, where $e_k(\theta_1,\theta_2,\dots, \theta_n)(\lambda)$ denotes the $k$-th elementary symmetric polynomial in $\theta_1,\theta_2,\dots, \theta_n$. Recall that the resultant of two polynomials, $p_1$ and $p_2$, can be determined as the product of $p_2$ evaluated at the roots of $p_1$ \cite{GKZ}, and we can calculate 
\begin{align}
\res_\mu(p(\mu,\lambda),p_\lambda(\mu,\lambda)) & \sim (-1)^{n(n-1)/2}\disc_\mu p(\mu,\lambda)(\theta_1\theta_2\cdots\theta_n\theta_1'\theta_2'\cdots\theta_n')(\lambda) \notag \\ 
&=(-1)^{n(n-1)/2+n+1}n^{-n}T^n\disc_\mu p(\mu,\lambda)\lambda^{1-n}\label{eqn:res(infty)}
\end{align}
for $\lambda\in i\mathbb{R}$ as $|\lambda|\to\infty$, where the prime means ordinary differentiation. Here the equality uses $\sum \omega_k=0$ and $\prod \omega_k=(-1)^{n+1}$. 

We can explicitly calculate the resultant for the limiting eigenvalue problem of \eqref{eqn:a} for $\lambda\in i\mathbb{R}$ as $|\lambda|\to\infty$ and, hence, for \eqref{eqn:a} itself at the leading order in $\lambda$. Referring to \eqref{def:res+}, we can determine the bifurcation index for the $L^2(\mathbb{R})$ essential spectrum of \eqref{eqn:a} away from the imaginary axis for $\lambda\in i\mathbb{R}$ as $|\lambda|\to\infty$. 
Since the characteristic polynomial for the limiting problem $\lambda v=\partial_x^nv$, for $n$ odd, will have $n$ distinct roots, the discriminant does not vanish, and neither does the resultant. Therefore, any bifurcations of the spectrum of \eqref{eqn:a} away from the imaginary axis will ultimately terminate along the imaginary axis towards $\pm i\infty$. This completes the proof.
\end{proof}

\begin{remark*}
Let $\lambda=i\nu^n$, where $\nu\gg1$, and we can calculate 
\begin{align} \label{eqn:res(infty)'}
\varPhi_n(\lambda)&\sim \prod_{1\leq j<k \leq n}(e^{\lambda^{1/n}\omega_j T}-e^{\lambda^{1/n}\w_k T})^2 \notag \\ 
&=\prod_{1\leq j<k \leq n}(2i)^{n(n-1)}\sin^2\Big(\nu Te^{\frac{1+2(j+k)\pi i}{2n}}\sin(\tfrac{j-k}{n}\pi)\Big),
\end{align}
which is real for $n$ odd. Indeed, we observe that all arguments of the sine functions either result in purely imaginary values or occur in pairs. When these pairs are multiplied together, they yield real values because they consist of complex numbers symmetric with respect to reflections across either the real or imaginary axis. Utilizing the identity
\[
\sin(a+ib)\sin(a-ib)=\tfrac12(\cosh(2b)-\cos(2a)),
\]
we can demonstrate that \eqref{eqn:res(infty)'} is real for $n$ odd. This also follows from the symmetry of the characteristic polynomial as well as Theorem~\ref{thm:asym}.
\end{remark*}

For example, when considering the generalized KdV equation (for $n=3$) and the Kawahara equation (for $n=5$), Theorem~\ref{thm:asym} implies that there are only a finite number of points along the imaginary axis where the $L^2(\mathbb{R})$ essential spectrum for the stability problem bifurcates away from the axis. In the case of even $n$, it is possible to derive a leading-order approximation of such bifurcations at $\lambda \in i\mathbb{R}$ and $|\lambda|\gg1$, provided that the limiting eigenvalue problem itself is Hamiltonian. Additionally, \eqref{eqn:res(infty)} and \eqref{eqn:res(infty)'} offer a method to calculate the bifurcation index up to leading order in $\lambda$ as $\lambda \to \pm i\infty$ for some lower-order eigenvalue problems. This can also be used to calculate the discriminant of the characteristic polynomial. However, it is important to note that these formulae can become quite intricate, and our approach has been to evaluate each equation on a case-by-case basis. Sections \ref{sec:asym3}, \ref{sec:asym4} and \ref{sec:asym5} provide more details.

\subsection{Numerical method}\label{sec:numerics}

In what follows we present a number of numerical experiments to illustrate these results. The numerics presented are done using two different techniques. The spectra of the linear operators are computed using the Fourier-Floquet-Hill method\cite{DeconinckKutz_2006}, a type of spectral method. The potentials are expanded in a Fourier series, and the operators expressed as operators on the sequence space $\ell_2$. These operators are then truncated, resulting in an $N\times N$ matrix eigenvalue problem, which we then solve across the entire range of the Floquet exponent, to obtain the complete spectrum. Typically $N=31$, encompassing wave numbers spanning from $-15$ to $15$. 

For the majority of examples, the periodic potential function can be conveniently represented using elementary functions or Jacobi elliptic functions, and their Fourier coefficients can be calculated using well-established analytic formulae (see, for instance, \cite{Kiper} for more details). However, in one instance involving the generalized KdV equation, a traveling wave solution cannot be expressed in terms of elliptic functions, and the Fourier coefficients are computed numerically. This involves solving the ODE governing periodic traveling waves, followed by numerical integration. 

To numerically compute the Floquet discriminant, which enables us to determine the algebraic multiplicity of the spectrum along the imaginary axis, as well as the bifurcation index for the spectrum away from the axis, we solve the ODEs corresponding to the linearized operator. 

Throughout the course of the numerical experiments, the spectrum is visualized using blue curves, while magenta lines parallel to the imaginary axis indicate intervals of maximal algebraic multiplicities. Dashed red curves represent the bifurcation index. In some instances, the axes may be interchanged to better highlight bifurcation points because the bifurcation index remains real along the imaginary axis.  

\section{Third order equations}\label{sec:3}

\subsection{The generalized KdV equation}\label{sec:KdV}

We begin our discussion by taking $n=3$ and the spectral problem for the generalized KdV equation
\begin{equation}\label{eqn:L(KdV)}
\lambda v=v_{xxx}+(Q(x)v)_x, \qquad \lambda\in\mathbb{C},
\end{equation}
where $Q(x)$ is a real-valued function satisfying $Q(x+T)=Q(x)$ for some $T>0$, the period. 
We do not impose additional assumptions such as evenness. Clearly, \eqref{eqn:L(KdV)} can be written in the form of \eqref{eqn:JH}, where 
\[
\boldsymbol{\mathcal{J}}=\partial_x\quad\text{and}\quad 
\boldsymbol{\mathcal{L}}=\partial_{xx}+Q(x),
\]
and the $L^2(\mathbb{R})$ essential spectrum of \eqref{eqn:L(KdV)} remains invariant under the transformations 
\[
\lambda \mapsto -\lambda\quad\text{and}\quad\lambda \mapsto \overline{\lambda}.
\]

We notice that the nonzero eigenvalues of \eqref{eqn:L(KdV)} coincide with the nonzero eigenvalues of 
\begin{equation}\label{eqn:L(KdV)0}
\lambda w=w_{xxx}+Q(x)w_x.
\end{equation}
Indeed, if $\lambda\neq 0$ is an eigenvalue of \eqref{eqn:L(KdV)} then the corresponding eigenfunction $v$ must have a zero mean over the period. Let $w_x=v$, where $w$ is periodic, and we can show that $w$ satisfies \eqref{eqn:L(KdV)0}. More generally, two operators $\boldsymbol{\mathcal{L}}_1\boldsymbol{\mathcal{L}}_2$ and $\boldsymbol{\mathcal{L}}_2\boldsymbol{\mathcal{L}}_1$ share common nonzero eigenvalues. Additionally, \eqref{eqn:L(KdV)0} is the negative adjoint of \eqref{eqn:L(KdV)}, which gives half of Hamiltonian symmetry. Particularly, the eigenvalues of \eqref{eqn:L(KdV)} and, hence, \eqref{eqn:L(KdV)0} are invariant with respect to reflection across the imaginary axis. 

We can reformulate \eqref{eqn:L(KdV)0} as
\begin{equation}\label{eqn:A(KdV)}
\mathbf{w}_x=
\begin{pmatrix} 0 & 1 & 0\\ 0 & 0 & 1 \\ \lambda  & -Q(x) & 0 \end{pmatrix}\mathbf{w}=:\mathbf{A}(x,\lambda)\mathbf{w},
\end{equation}
and we verify that \text{(A1)} and \text{(A2)} hold for $\mathbf{B}(\lambda)=\begin{pmatrix}
-\lambda & 0 & 0 \\ 0 & 0 & -1 \\ 0 & 1 & 0
\end{pmatrix}$. 
Additionally, we verify that $\tr(\mathbf{A}(x,\lambda))=0$ for all $x\in\mathbb{R}$ and $\lambda\in\mathbb{C}$. We define the monodromy matrix of \eqref{eqn:A(KdV)} as
\[
\mathbf{M}(\lambda)=\mathbf{W}(T,\lambda), \quad\text{where} \quad 
\mathbf{W}_x=\mathbf{A}(x,\lambda)\mathbf{W}
\quad\text{and}\quad \mathbf{W}(0,\lambda)=\mathbf{I}_3,
\]
$\mathbf{I}_3$ is the $3\times3$ identity matrix. We define the characteristic polynomial of the monodromy matrix of \eqref{eqn:A(KdV)} as 
\[
p(\mu,\lambda)=\det(\mathbf{M}(\lambda)-\mu\mathbf{I}_3).
\]
The monodromy matrix and the characteristic polynomial inherit symmetry from \eqref{eqn:A(KdV)}. Particularly,  
\[
p(\mu,\lambda)=-\mu^3+\boldsymbol{f}(\lambda)\mu^2-\boldsymbol{f}(-\lambda)\mu+1
\quad \text{for $\lambda \in \mathbb{C}$}, 
\]
where $\boldsymbol{f}(\lambda)=\tr(\mathbf{M}(\lambda))$, and
\[
p(\mu,\lambda)=-\mu^3+\boldsymbol{f}(\lambda)\mu^2-\overline{\boldsymbol{f}(\lambda)}\mu+1\quad\text{for $\lambda \in i\mathbb{R}$}.
\]
We remark that $\boldsymbol{f}(-\lambda)=\boldsymbol{f}(\overline{\lambda})=\overline{\boldsymbol{f}(\lambda)}$ for $\lambda \in i\mathbb{R}$.

Our interest lies in the $L^2(\mathbb{R})$ essential spectrum of \eqref{eqn:L(KdV)0} along the imaginary axis, which holds significant importance when investigating the stability and instability of periodic traveling waves of the generalized KdV equation and related equations. Theorem~\ref{thm:spec3} establishes that the monodromy matrix of \eqref{eqn:A(KdV)} at $\lambda \in i\mathbb{R}$ has three simple eigenvalues on the unit circle, that is, $\lambda$ belongs to the spectrum with an algebraic multiplicity three, if and only if the Floquet discriminant $\boldsymbol{f}(\lambda)$, defined in \eqref{def:disc3}, lies within the deltoidal region bounded by the deltoid curve in \eqref{def:deltoid}. Furthermore, Theorem~\ref{thm:bifur3} demonstrates that if the spectrum of \eqref{eqn:L(KdV)0} bifurcates at $\lambda \in i \mathbb{R}$ away from the imaginary axis in a transversal manner then the bifurcation index $\varPhi_3(\lambda)$, defined in \eqref{def:res3}, becomes zero. We will proceed with numerical experiments to validate these analytical results. 
 
\subsection{Numerical experiments for equations of KdV type}\label{sec:KdV numerics}
 
We begin our numerical experiments with the Mathieu equation 
\begin{equation}\label{eqn:MKdV}
\lambda w=w_{xxx}+(4+5\cos(x))w_x,\qquad \lambda\in\mathbb{C}.
\end{equation}
Although it may not seem directly linked to a stability problem for a periodic traveling wave of the generalized KdV equation, as far as we are aware, \eqref{eqn:MKdV} does represent a spectral problem that exhibits the required symmetry.
 
\begin{figure}[htbp]
\centering
\includegraphics[width=0.45\textwidth]{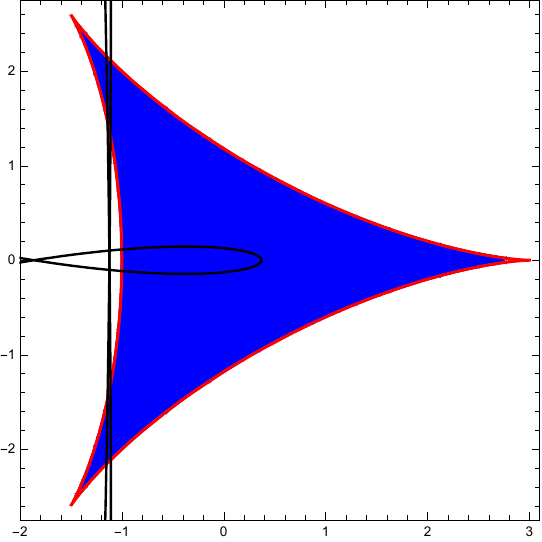}
\caption{The deltoidal region (blue) and the trajectory of the Floquet discriminant for \eqref{eqn:MKdV} (black).}
\label{fig:Mathieu1}
\end{figure}

Figure~\ref{fig:Mathieu1} depicts the deltoidal region in $\mathbb{R}^2(=\mathbb{C})$,  enclosed by the deltoid curve defined in \eqref{def:deltoid}, together with the trajectory of the Floquet discriminant $\boldsymbol{f}(\lambda)$ for \eqref{eqn:MKdV} as $\lambda$ varies over the imaginary axis. 
Our numerical observation reveals that for $\lambda \in i\mathbb{R}$ in the interval approximately $(-.4462i, .4462i)$ $\cup$  $\pm (6.0153i, 6.022i)$ $ \cup$ $ \pm (6.0451i, 6.0509i)$, 
the algebraic multiplicity is three. For other values of $\lambda \in i\mathbb{R}$, the multiplicity appears to be one.

\begin{figure}[htbp]
\centering
\includegraphics[width=0.45\textwidth]{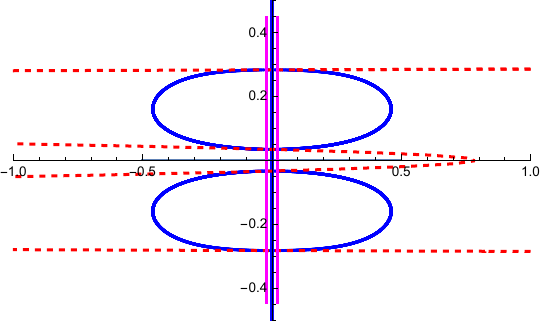}
\includegraphics[width=0.4\textwidth]{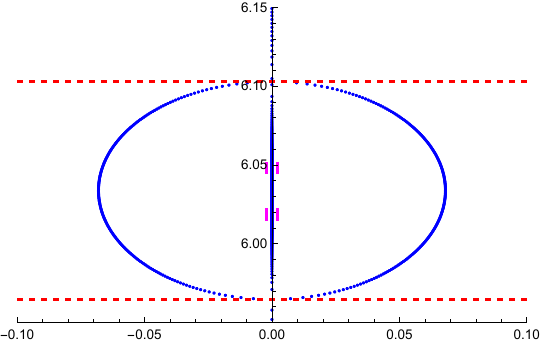}
\caption{The blue curves represent the numerically computed spectrum of \eqref{eqn:MKdV} for $\lambda \in (-0.5i,0.5i)$ (left) and $(5.95i, 6.12i)$ (right). The magenta lines indicate intervals of algebraic multiplicity three, and the dashed red curves intersect the imaginary axis where the bifurcation index is zero.}\label{fig:Mathieu2}
\end{figure}

Moving on to Figure~\ref{fig:Mathieu2}, the blue curves are the numerically computed $L^2(\mathbb{R})$ essential spectrum of \eqref{eqn:MKdV}. Additionally, the magenta lines, running parallel to the imaginary axis, indicate intervals where $\boldsymbol{f}(\lambda)$, $\lambda \in i\mathbb{R}$, resides within the deltoidal region and, consequently, $\lambda$ is in the spectrum with an algebraic multiplicity three. The dashed red curves, representing the graph of the bifurcation index $\varPhi_3(\lambda)$, intersect the imaginary axis, identifying points where bifurcations of the spectrum away from the imaginary axis are anticipated. Note that these bifurcations manifest in regions of multiplicity three (on the left) or regions of multiplicity one (on the right).

We turn our attention to the generalized KdV equation 
\begin{equation}\label{eqn:gKdV}
u_t+u_{xxx}+(u^k)_x=0, \qquad \text{$k\geq2$ an integer}. 
\end{equation}
For $k=2$ (the KdV equation) and $k=3$ (the mKdV equation), \eqref{eqn:gKdV} is integrable via the inverse scattering transform, which offers a rigorous method for determining the stability and instability of periodic traveling waves. There has, of course, been a great deal of work aimed at understanding the stability of generalized KdV traveling waves, both solitary waves and periodic wavetrains: see for instance\cite{DeconinckKapitula10,Johnson09nonlinear,MaddocksSachs93,MartelMerle01,McKean77,ScharfWreszinski81}.

A traveling wave solution of \eqref{eqn:gKdV} takes the form $u(x,t)=\phi(x-ct)$ for some $c\neq 0,\in\mathbb{R}$, the wave speed, and it satisfies by quadrature
\begin{equation*}\label{eqn:c(gKdV)}
\phi_x^2=2E+2a\phi+c\phi^2-\frac{2}{k+1}\phi^{k+1}=:G(\phi;c,a,E),
\end{equation*}
where  $a$ and $E$ are real constants. Here we assume that $\phi$ is periodic. Note that $G(\phi;c,a,E)$ is a polynomial in $\phi$, and is elliptic for $k=2$ and $3$ while it becomes hyperelliptic in general for $k\geq 4$, an integer. 

Linearizing \eqref{eqn:gKdV} about a periodic traveling wave $\phi(x; c, a, E)$ in the frame of reference moving at the speed $c$, we arrive at
\begin{equation}\label{eqn:L(gKdV)}
\lambda v-cv_x+v_{xxx}+(k\phi^{k-1}v)_x=0,\qquad \lambda\in\mathbb{C}.
\end{equation}
At $\lambda=0$, it is possible to explicitly calculate the associated monodromy matrix and the characteristic polynomial, and the followings hold true:
\begin{itemize}
\item The Floquet discriminat $\boldsymbol{f}(0)=\tr(\mathbf{M}(0))= 3$, where $\mathbf{M}(\lambda)$ denotes the monodromy matrix. Specifically, $1$ is an eigenvalue of $\mathbf{M}(0)$ with an algebraic multiplicity three and a geometric multiplicity two.
\item The discriminant of the characteristic polynomial $\Delta_3(0)=0$, where $\Delta_3(\lambda)$ is defined in \eqref{def:disc3}.
\item The bifurcation index $\varPhi_3(0)=0$, where $\varPhi_3(\lambda)$ is in \eqref{def:res3}.  
\end{itemize}
As a result, Theorem~\ref{thm:bifur3} becomes inconclusive and one must conduct a detailed modulation calculation to determine whether the $L^2(\mathbb{R})$ essential spectrum of \eqref{eqn:L(gKdV)} bifurcates at $0\in\mathbb{C}$ away from the imaginary axis. Interested readers can refer to \cite{JThesis,BJ,BJK1,JMMP1,BH}, among many others, for further elaboration. Our emphasis here is on bifurcations of the spectrum at $\lambda \neq 0, \in i\mathbb{R}$.
 
\begin{figure}[htbp]
\centering
\includegraphics[width=0.45\textwidth]{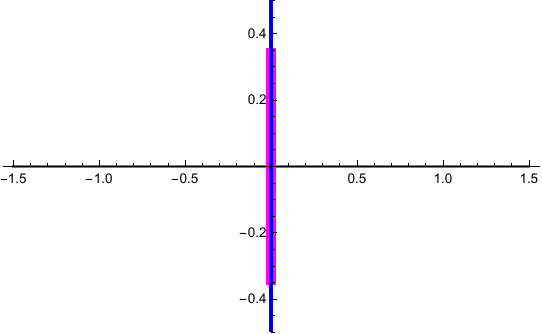}
\caption{The numerically computed spectrum of \eqref{eqn:L(gKdV)} for $k=3$ (the mKdV equation) with $c=1$, $a=\tfrac14$, and $E=0$, suggesting spectral stability. The spectrum has an algebraic multiplicity three in the interval  $\approx (-0.354i, 0.354i)$, while the multiplicity is one elsewhere along $i\mathbb{R}$. The only zero of the bifurcation index is at $0$, where the spectrum does not seem to bifurcate away from the imaginary axis.}\label{fig:mKdV1}
\end{figure}

For example, Figure~\ref{fig:mKdV1} depicts the numerically computed $L^2(\mathbb{R})$ essential spectrum of \eqref{eqn:L(gKdV)} for $k=3$---namely, the linearized mKdV equation---about a periodic traveling wave with parameters $c=1$, $a=\tfrac14$, and $E=0$. Recall \cite{BJK1} that if $G(\phi; c,a,E)=2E+2a\phi+c\phi^2-\tfrac12\phi^{4}$ has four real roots then the corresponding periodic traveling wave is modulationally stable and, correspondingly, the spectrum lies along the imaginary axis with an algebraic multiplicity three near $0\in\mathbb{C}$. We verify that $G(\phi;1,\tfrac14,0)$ indeed has four real roots. Our numerical findings corroborate this, revealing that for $\lambda \in i\mathbb{R}$ in the interval $(-\lambda_0i, \lambda_0i)$, where $\lambda_0$ is approximately $0.354$, the algebraic multiplicity is three. In the remaining intervals of the imaginary axis, the multiplicity appears to be one. Additionally, apart from the zero at $0 \in i\mathbb{R}$, no other zeros of the bifurcation index are detected. We numerically observe no bifurcations of the spectrum near $0\in\mathbb{C}$ away from the imaginary axis,
implying the spectral stability of the underlying periodic traveling wave. Indeed, all numerically computed eigenvalues have a maximal real part $< 8\times10^{-10}$, well within the numerical accuracy. 
 
\begin{figure}[htbp]
\centering
\includegraphics[width=0.45\textwidth]{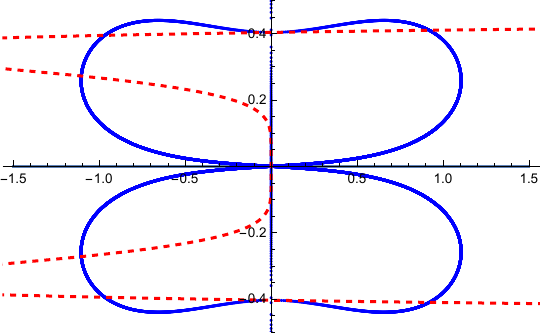}
\caption{The numerically computed spectrum of \eqref{eqn:L(gKdV)} for $k=3$ with $c=1$, $a=\tfrac15$, and $E=\tfrac25$, suggesting spectral instability. The spectrum appears to bifurcate away from the imaginary axis at $0\in\mathbb{C}$ as well as $\approx \pm0.4i$. The dashed red curves represent the graph of the bifurcation index rotated by $90^\circ$, with intersections along the imaginary axis indicating bifurcation points.}   
\label{fig:mKdV2}
\end{figure}

Figure~\ref{fig:mKdV2} shows the numerically computed $L^2(\mathbb{R})$ essential spectrum of \eqref{eqn:L(gKdV)} for $k=3$ about a periodic traveling wave for $c=1$, $a=\tfrac15$, and $E=\tfrac25$. A direct calculation reveals that $G(\phi; 1, \tfrac15, \tfrac25)$ has two real roots and two complex roots, whereby it follows from \cite{BJK1} that the underlying periodic traveling wave is modulationally unstable. Indeed, we numerically observe that a spectral curve emerges from $0\in\mathbb{C}$ along the imaginary axis, accompanied by two additional spectral curves emerging in a non-tangential manner, consistent with underlying Hamiltonian symmetry. Additionally, we numerically observe that the algebraic multiplicity is one for the entire spectrum. The dashed red curves represent the graph of the bifurcation index rotated by $90^\circ$. The intersections of these curve with the imaginary axis identify the points where bifurcations of the spectrum away from the imaginary axis are anticipated. The agreement between the numerical findings and analytical predictions is strong.  
 
\begin{figure}[htbp]
\centering
\includegraphics[width=0.45\textwidth]{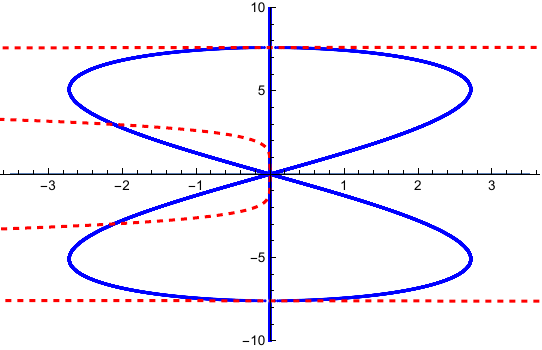}
\caption{The numerically computed spectrum of \eqref{eqn:L(gKdV)} for $k=4$ with $c=1$, $a=2$, and $E=0$. The spectrum appears to bifurcate away from the imaginary axis at $0\in\mathbb{C}$ and $\approx \pm 7.6i$. The dashed red curves represent the graph of the bifurcation index rotated by $90^\circ$, with intersections along the imaginary axis indicating bifurcation points.}   
\label{fig:mKdV3}
\end{figure}

Lastly, Figure~\ref{fig:mKdV3} shows the numerically computed $L^2(\mathbb{R})$ essential spectrum of \eqref{eqn:L(gKdV)} for $k=4$ about a periodic traveling wave for $c=1$, $a=2$, and $E=0$, suggesting mondulational instability. Indeed, our numerical observation reveals that bands of unstable spectrum bifurcate from $0\in\mathbb{C}$, exhibiting a characteristic ``figure eight'' curve that is consistent with underlying Hamiltonian symmetry. These bands of unstable spectrum then rejoin the imaginary axis at approximately $\pm 7.6i$, the zeros of the bifurcation index. 

\subsection{The generalized BBM equation}\label{sec:BBM}

Another illustrative example is the generalized BBM equation
\begin{equation}\label{eqn:gBBM}
u_t-u_{xxt}+u_x+g(u)_x=0
\end{equation} 
for some nonlinearity $g$ or, equivalently,  
\begin{align}\label{eqn:gBBM'}
u_t-u_{xxt}+cu_{xxx}+(1-c)u_x+g(u)_x=0
\end{align}
in the frame of reference moving at some $c\neq0, \in\mathbb{R}$, the wave speed. A traveling wave solution of \eqref{eqn:gBBM} corresponds to a stationary solution of \eqref{eqn:gBBM'}, and it satisfies
\[
c\phi_{xxx}+(1-c)\phi_x+(g(\phi))_x=0.
\]
Here we assume that $\phi(x+T)=\phi(x)$ for some $T>0$, the period. 

Linearizing \eqref{eqn:gBBM'} about a stationary solution $\phi(x)$, we arrive at 
\begin{equation}\label{eqn:L(gBBM)}
\lambda v-\lambda v_{xx}+cv_{xxx}+(1-c)v_x+(g'(\phi)v)_x=0, \qquad \lambda\in\mathbb{C},
\end{equation}
where the prime means ordinary differentiation. Introducing
\[
v_2=-\lambda v_x+cv_{xx}+(1-c)v+g(\phi)v \quad\text{and}\quad
v_3=-\lambda v+cv_x, 
\]
we can reformulate \eqref{eqn:L(gBBM)} as 
\begin{align}\label{eqn:A(gBBM)}
\begin{pmatrix}v \\ v_2 \\ v_3 \end{pmatrix}_x
& = \begin{pmatrix}
\tfrac{\lambda}{c} & 0 & \tfrac{1}{c} \\
-\lambda & 0 & 0 \\ 
-Q(x) & 1 & 0 
\end{pmatrix}
\begin{pmatrix}v \\ v_2 \\ v_3 \end{pmatrix}
=:\mathbf{A}(x;\lambda)\mathbf{v},
\end{align}
where $Q(x)=g'(\phi(x))+1-c$, and we verify that \text{(A1)} and \text{(A2)} hold for 
\begin{equation} \label{def:B(gBBM)}
\mathbf{B}(\lambda)=\begin{pmatrix}
-\lambda & 0 & 1 \\ 0 & \tfrac1\lambda & 0 \\ - 1 & 0 & 0 
\end{pmatrix}, 
\end{equation}
provided that $\lambda\neq0$. However, it is important to note that $\tr(\mathbf{A}(x,\lambda))\neq0$ for $\lambda \neq 0$. 
A spectral analysis for \eqref{eqn:L(gBBM)} follows a similar approach to that in Section~\ref{sec:KdV} for the generalized KdV equation. Further details on a modulational calculation can be found in \cite{J2} and \cite{JThesis}, among others.

Let $\mathbf{M}(\lambda)$ denote the monodromy matrix of \eqref{eqn:A(gBBM)}. The characteristic polynomial of the monodromy matrix (see Johnson\cite{J2}) takes the form
\[
p(\mu,\lambda)=\det(\mathbf{M}(\lambda)-\mu\mathbf{I}_3)=-\mu^3+\boldsymbol{f}(\lambda)\mu^2-\boldsymbol{f}(-\lambda)e^{\frac{\lambda T}{c}}\mu+e^{\frac{\lambda T}{c}},
\qquad \lambda \in \mathbb{C},
\]
where $\boldsymbol{f}(\lambda)=\tr(\mathbf{M}(\lambda))$. The discriminant of the characteristic polynomial is no longer real for $\lambda \in i\mathbb{R}$. Instead, 
\begin{align*}
\Delta_3(\lambda):=&e^{-\frac{2\lambda T}{c}}\disc_\mu p(\mu,\lambda) \\
=&\boldsymbol{f}(\lambda)^2\boldsymbol{f}(-\lambda)^2-4\boldsymbol{f}(\lambda)^3e^{-\frac{\lambda T}{c}}-4\boldsymbol{f}(-\lambda)^3e^{\frac{\lambda T}{c}}+18\boldsymbol{f}(\lambda)\boldsymbol{f}(-\lambda)-27
\end{align*}
is real for $\lambda \in i\mathbb{R}$. We observe that $\Delta_3(\lambda)=0$, $\lambda \in i\mathbb{R}$, when $\boldsymbol{f}(\lambda)$ lies on the deltoid curve, defined in \eqref{def:deltoid}, rotated through an angle of $\tfrac{\lambda T}{3c}$, whereby $\Delta_3$ can be used to determine the algebraic multiplicity of the $L^2(\mathbb{R})$ essential spectrum of \eqref{eqn:L(gBBM)} along the imaginary axis. The resultant of $p(\mu,\lambda)$ and $p_\lambda(\mu,\lambda)$ is likewise no longer real for $\lambda \in i\mathbb{R}$. Instead,
\begin{align*}
\varPhi_3(\lambda):=&e^{-\frac{3\lambda T}{c}}\res_\mu(p(\mu,\lambda), p_\lambda(\mu,\lambda)) \\
=&e^{-\frac{\lambda T}{c}}\boldsymbol{f}'(\lambda)^3+e^{\frac{\lambda T}{c}}\boldsymbol{f}'(-\lambda)^3
+\boldsymbol{f}(\lambda)\boldsymbol{f}'(-\lambda)^2\boldsymbol{f}'(\lambda)+\boldsymbol{f}(-\lambda)\boldsymbol{f}'(\lambda)^2\boldsymbol{f}'(-\lambda) \\
&+e^{\frac{\lambda T}{c}}\frac{T^2}{c^2}(\boldsymbol{f}(\lambda)^2\boldsymbol{f}'(\lambda)+\boldsymbol{f}(-\lambda)^2\boldsymbol{f}'(-\lambda)) \\
&-e^{\frac{\lambda T}{c}}\frac{2T}{c}(\boldsymbol{f}(\lambda)\boldsymbol{f}'(\lambda)^2+\boldsymbol{f}(-\lambda)\boldsymbol{f}'(-\lambda)^2) \\
&+\frac{T^3}{c^3}(\boldsymbol{f}(\lambda)\boldsymbol{f}(-\lambda)+1)
+\frac{T^2}{c^2}(\boldsymbol{f}(\lambda)\boldsymbol{f}'(-\lambda)+\boldsymbol{f}'(\lambda)\boldsymbol{f}(-\lambda)) \\
&-\frac{T}{c}(\boldsymbol{f}(\lambda)\boldsymbol{f}'(\lambda)\boldsymbol{f}(-\lambda)\boldsymbol{f}'(-\lambda)+3\boldsymbol{f}'(\lambda)\boldsymbol{f}'(-\lambda))
\end{align*}
is real for $\lambda \in i\mathbb{R}$, and if the spectrum of \eqref{eqn:L(gBBM)} bifurcates at $\lambda \in i\mathbb{R}$ away from the imaginary axis then $\varPhi_3(\lambda)=0$. 

We remark that the exponential prefactor in the definitions of $\Delta_3$ and $\varPhi_3$ does not affect the roots of $\disc_\mu(p(\mu,\lambda))$ and $\res_\mu(p(\mu,\lambda), p_\lambda(\mu,\lambda))$, and it is included because, in practice, working with real-valued polynomials aids in root finding, compared with dealing with complex-valued polynomials.
Additionally, as was the case for the generalized KdV equation, under some additional assumptions, $\varPhi_3(\lambda)=0$ also serves as a sufficient condition for the spectrum to bifurcate away from the imaginary axis. 

Alternatively, let 
\[
\mathbf{v}_0=e^{-\frac{\lambda x}{3c}}\mathbf{v},
\]
and we can demonstrate that if $\mathbf{v}$ satisfies \eqref{eqn:A(gBBM)} then $\mathbf{v}_0$ satisfies
\begin{equation}\label{eqn:A0(gBBM)}
{\mathbf{v}_0}_x=\left(\mathbf{A}(x,\lambda)-\frac{\lambda}{3c}\mathbf{I}_3\right)\mathbf{v}_0=:\mathbf{A}_0(x,\lambda)\mathbf{v}_0.
\end{equation}
We verify that \text{(A1)} and \text{(A2)} remain valid for \eqref{def:B(gBBM)}. An important difference here is that  $\tr(\mathbf{A}_0(x,\lambda))=0$ for all $x\in\mathbb{R}$ for all $\lambda\in\mathbb{C}$. Let $\mathbf{M}_0(\lambda)$ denote the monodromy matrix of \eqref{eqn:A0(gBBM)} and the characteristic polynomial of the monodromy matrix is
\begin{align*}
&p_0(\mu,\lambda)
=-\mu^3+\boldsymbol{f}_0(\lambda)\mu^2-\boldsymbol{f}_0(-\lambda))\mu+1 
\quad \text{for $\lambda\in\mathbb{C}$},
\intertext{where $\boldsymbol{f}_0(\lambda)=\tr(\mathbf{M}_0(\lambda))$, and}
&p_0(\mu,\lambda)=-\mu^3+\boldsymbol{f}_0(\lambda)\mu^2-\overline{\boldsymbol{f}_0(\lambda)}\mu+1 
\quad \text{for $\lambda\in i\mathbb{R}$}.
\end{align*}

Note that if $p(\mu,\lambda)=0$ then $p_0\big(e^{\frac{\lambda T}{3c}}\mu,\lambda\big)=0$. In other words, a root of $p_0(\cdot, \lambda)$ for $\lambda \in i\mathbb{R}$ is a rotation of a root of $p(\cdot, \lambda)$. Consequently, the number of eigenvalues of $\mathbf{M}(\lambda)$ on the unit circle is the same as the number of eigenvalues of $\mathbf{M}_0(\lambda)$ on the unit circle. Note that 
\[
\boldsymbol{f}_0(\lambda)=\tr(\mathbf{M}_0(\lambda))=e^{\frac{\lambda T}{3c}}\tr(\mathbf{M}(\lambda))=e^{\frac{\lambda T}{3c}}\boldsymbol{f}(\lambda),
\]
and we calculate
\begin{align*}
\res_\mu(&p_0(\mu,\lambda), {p_0}_\lambda(\mu,\lambda)) \\
=&e^{\frac{\lambda T}{c}}\left(\frac{T}{3c}\boldsymbol{f}(\lambda) +\boldsymbol{f}'(\lambda)\right)^3-e^{-\frac{\lambda T}{c}}\left(\frac{T}{3c}\boldsymbol{f}(-\lambda)+\boldsymbol{f}'(-\lambda)\right)^3 \\
&+\boldsymbol{f}(\lambda)\left(\frac{T}{3c}\boldsymbol{f}(-\lambda)+\boldsymbol{f}'(-\lambda)\right)^2
\left(\frac{T}{3c}\boldsymbol{f}(\lambda)+\boldsymbol{f}'(\lambda)\right) \\ 
&-\boldsymbol{f}(-\lambda)\left(\frac{T}{3c}\boldsymbol{f}(\lambda)+\boldsymbol{f}'(\lambda)\right)^2\left(\frac{T}{3c}\boldsymbol{f}(-\lambda)+\boldsymbol{f}'(-\lambda)\right).
\end{align*}
We deduce that $\res_\mu(p(\mu,\lambda), {p}_\lambda(\mu,\lambda))=0$ for $\lambda\in i\mathbb{R}$ if and only $\res_\mu(p_0(\mu,\lambda), {p_0}_\lambda(\mu,\lambda))=0$. 

Therefore, Theorem~\ref{thm:spec3} establishes that the monodromy matrix of \eqref{eqn:A(gBBM)} at $\lambda \in i\mathbb{R}$ has three simple eigenvalues on the unit circle, that is, $\lambda$ belongs to the $L^2(\mathbb{R})$ essential spectrum with an algebraic multiplicity three if and only if $\Delta_3(\lambda)=e^{-\frac{2\lambda T}{c}}\disc_\mu p(\mu,\lambda)$ or $\disc_\mu(p_0(\mu,\lambda))<0$. The latter holds when the Floquet discriminant for \eqref{eqn:A0(gBBM)} lies within the deltoidal region bounded by the deltoid curve in \eqref{def:deltoid}. Furthermore, Theorem~\ref{thm:bifur3} demonstrates that if the spectrum of \eqref{eqn:L(KdV)0} bifurcates at $\lambda \in i \mathbb{R}$ away from the imaginary axis in a transversal manner then the bifurcation index $\varPhi_3(\lambda)=e^{-\frac{3\lambda T}{c}}\res_\mu(p(\mu,\lambda), p_\lambda(\mu,\lambda))$ or $\res_\mu(p_0(\mu,\lambda), {p_0}_\lambda(\mu,\lambda))$ becomes zero. We will present numerical experiments to validate these analytical findings.

\begin{remark*}
When $\lambda=0$, we may assume without loss of generality that $c=1$, and \eqref{eqn:L(gBBM)} simplifies by quadrature into the Hill's equation 
\[
v_{xx}+Q(x)v=v_0(x),
\]
where $v_0(x)$ represents the (time-independent) initial value of $v$. 
\end{remark*} 

\subsection{Numerical experiments for equations of BBM type}\label{sec:BBM numerics}

\begin{figure}[htbp]
\centering
\includegraphics[width=0.45\textwidth]{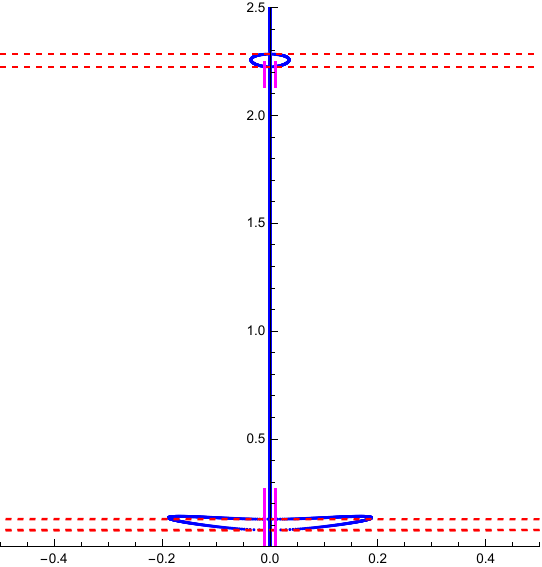}
\caption{The numerically computed spectrum for \eqref{eqn:MBBM}. The magenta lines indicate intervals of algebraic multiplicity three, and the dashed red lines intersect the imaginary axis where the bifurcation index becomes zero.}
\label{fig:MBBM}
\end{figure}

We begin our numerical experiments with
\begin{equation}\label{eqn:MBBM}
\lambda(v-v_{xx})+v_{xxx}+((4+5\cos(x))v)_x=0,\qquad \lambda\in\mathbb{C}.
\end{equation}
Similar to the Mathieu equation, \eqref{eqn:MBBM} may not seem to correspond to a stability problem for a periodic traveling wave of the generalized BBM equation, as far as we are aware, but it serves as a valuable test case. 

In Figure~\ref{fig:MBBM}, the blue curves are the numerically computed $L^2(\mathbb{R})$ essential spectrum of \eqref{eqn:MBBM}. It is noteworthy that the imaginary axis is included in the spectrum. Additionally, the magenta lines, running parallel to the imaginary axis, indicate intervals where the spectrum has an algebraic multiplicity three. We numerically observe that these intervals are approximately $(-2.25 i,-2.13 i)$, $(-0.27 i,0.27 i)$ and $(2.13 i,2.25 i)$. The multiplicity appears to be one in the remaining intervals of the imaginary axis. 
The dashed red lines, representing the graph of the bifurcation index rotated by $90^\circ$, intersect the imaginary axis, identifying the zeros of the bifurcation index where bifurcations of the spectrum away from the imaginary axis are anticipated. Our numerical findings corroborate this, revealing two isolas in the upper half-plane, approximately within the ranges of $0.77 i$ to $0.128 i$ as well as $2.225 i$ to $2.285 i$.


\begin{figure}[htbp]
\centering
\includegraphics[width=0.45\textwidth]{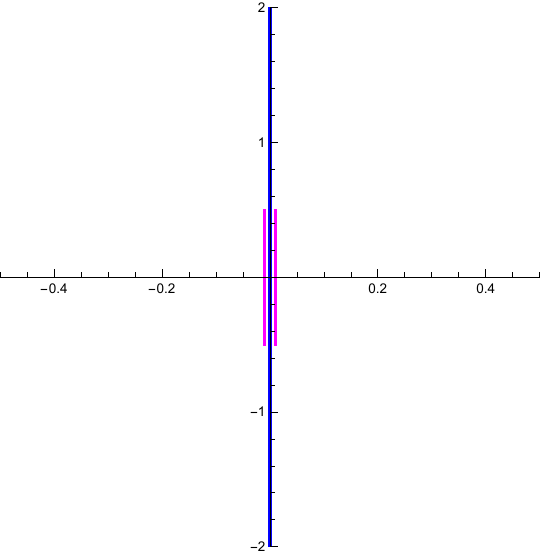}
\caption{The numerically computed spectrum of the linearized operator of \eqref{eqn:BBM} about \eqref{def:phi(BBM)}, suggesting spectral stability. The spectrum has an algebraic multiplicity three in the interval $\approx (-0.5i, 0.5i)$, while the multiplicity appears to be one elsewhere along $i\mathbb{R}$. The only zero of the bifurcation index occurs at $\lambda=0$, where bifurcation of the spectrum away from the imaginary axis does not seem to occur.} 
\label{fig:BBM}
\end{figure}

We turn our attention to the BBM equation
\begin{equation}\label{eqn:BBM}
u_t-u_{xxt}+u_x+uu_x=0.
\end{equation}
A traveling wave solution of \eqref{eqn:BBM} takes the form $u(x,t)=\phi(x-ct)$ for some $c\neq0, \in \mathbb{R}$, the wave speed, and it satisfies by quadrature
\[
c\phi_x^2=2E+2a\phi+(1-c)\phi^2-\tfrac13\phi^3,
\]
where $a$ and $E$ are real constants. Here we assume that $\phi$ is periodic. 

For example, Figure~\ref{fig:BBM} depicts the numerically computed $L^2(\mathbb{R})$ essential spectrum of the linearized operator of \eqref{eqn:BBM} about a periodic traveling wave for $c=1$, $a=\tfrac76$ and $E=-1$, which can be expressed in closed form as
\begin{equation}\label{def:phi(BBM)}
\phi(x)=1+\operatorname{cn}^2\Big(\sqrt{\tfrac{5}{12}}x;\tfrac15\Big).
\end{equation}
Here ``$\operatorname{cn}$'' represents a Jacobi elliptic function, the elliptic cosine function. It is worth noting that $\phi$ oscillates between a minimum of $\phi=1$ and a maximum of $\phi=2$ with a period $2\sqrt{\tfrac{12}{5}}K\left(\tfrac15\right)\approx 5.142$, where $K$ is the complete elliptic integral of the first kind.
We numerically observe that for $\lambda \in i\mathbb{R}$ in the interval $(-\lambda_0i,\lambda_0i)$, where $\lambda_0$ is approximately $0.5$, the algebraic multiplicity is three. Elsewhere along the imaginary axis, the multiplicity appears to be one. Additionally, apart from the zero at $0\in i\mathbb{R}$, no other zeros of the bifurcation index are found. We numerically observe no bifurcations of the spectrum near $0\in\mathbb{C}$ away from the imaginary axis, implying the spectral stability of the underlying periodic traveling wave. Indeed, all numerically computed eigenvalues have a maximal real part $<3.2\times10^{-8}$. 

\begin{figure}
\centering
\includegraphics[width=0.4\textwidth]{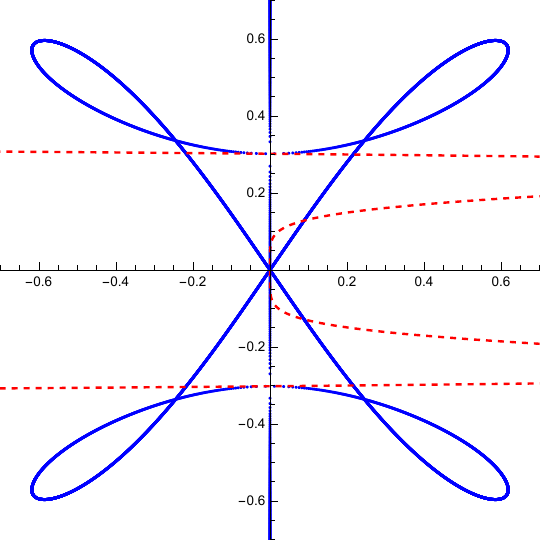}
\caption{The numerically computed spectrum of the linearized operator of \eqref{eqn:mBBM} about \eqref{def:phi(mBBM)} for $m=\tfrac12+\tfrac{\sqrt{5}}{10}$. The spectrum has an algebraic multiplicity one except at $0\in\mathbb{C}$. The bifurcation index has zeros at $0$ and as well as $\approx \pm 0.3 i$. }   
\label{fig:mBBM}
\end{figure}
 
Lastly, let's consider the (focusing) mBBM equation
\begin{equation}\label{eqn:mBBM}
u_t-u_{xxt}+(u^3)_x=0
\end{equation}
and a cnoidal solution
\begin{equation}\label{def:phi(mBBM)}
u(x,t)=\sqrt{\frac{2m}{2m-1}}\operatorname{cn}\left(\frac{x-t}{\sqrt{2m-1}};m\right), \qquad m>1/2. 
\end{equation}
Although our numerical investigation is not exhaustive, the stability and instability of periodic traveling waves of \eqref{eqn:mBBM} seem to resemble those of the mKdV equation. Specifically, the snoidal solutions (for the defocusing equation) and dnoidal solutions (for the focusing equation) appear stable while the cnoidal solutions seem to exhibit instability. 

Figure~\ref{fig:mBBM} shows the numerically computed $L^2(\mathbb{R})$ essential spectrum of the linearized operator of \eqref{eqn:mBBM} about a cnoidal solution \eqref{def:phi(mBBM)} for $m=\tfrac12+\tfrac{\sqrt{5}}{10}\approx.7236$. Modulational instability is evident, with the spectrum bifurcating away from the imaginary axis at $0\in\mathbb{C}$ and approximately $\pm 0.3 i$. These bifurcation points align with the three zeros of the bifurcation index.

\subsection{Spectrum along the imaginary axis towards $\pm i\infty$}\label{sec:asym3}

Applying Theorem~\ref{thm:asym} to \eqref{eqn:L(gKdV)}, where $\lambda=i\nu^3$ and $|\nu| \gg1$, we obtain
\[
\Delta_3(i\nu^3) \sim 16 \sinh^2\left(\tfrac{\sqrt{3}}{2}\nu T\right)
\left(\cos\left(\tfrac32\nu T\right)-\cosh\left(\tfrac{\sqrt{3}}{2}\nu T\right)\right)^2>0,
\]
provided that $\nu \neq 0$. Consequently, we deduce from Theorem~\ref{thm:spec3} that $\lambda \in i\mathbb{R}$, $|\lambda|\gg1$, is in the $L^2(\mathbb{R})$ essential spectrum of \eqref{eqn:L(gKdV)} with an algebraic multiplicity one. Furthermore,
\begin{align*}
\varPhi_3(i\nu^3)&\sim-\frac{1}{27}\frac{T^3}{\nu^6}\Delta_3(i\nu^3) \\ 
&=-\frac{16}{27}\frac{T^3}{\nu^6}\sinh^2\left(\tfrac{\sqrt{3}}{2}\nu T\right)
\left(\cos\left(\tfrac32\nu T\right)-\cosh\left(\tfrac{\sqrt{3}}{2}\nu T\right)\right)^2<0.
\end{align*}
Consequently, Theorem~\ref{thm:bifur3} implies that only a finite number of points along the imaginary axis can exhibit transversal bifurcation of the spectrum away from the axis.

Note that the linearization of \eqref{eqn:gBBM} does not adhere to the form \eqref{eqn:a}, which prevents us from directly applying Theorem \ref{thm:asym}. Had it been possible to construct a full-dimensional and linearly independent set of asymptotic solutions, either through the WKB method or alternative means, we could have followed the argument in a similar manner to the proof of Theorem \ref{thm:asym} to derive leading-order expressions for the discriminant of the characteristic polynomial and the bifurcation index. However, our attempts to identify such a set of solutions were unsuccessful, leaving us unable to generate the required formulae.  
 
\section{Fourth order equations}\label{sec:4}
 
\subsection{The nonlinear Schr\"odinger equation}\label{sec:NLS}

We turn our attention to $n=4$ and the spectral problem for the nonlinear Schr\"odinger equation 
\begin{subequations}\label{eqn:L(NLS)}
\begin{equation}\label{def:NLS}
\lambda\begin{pmatrix}
v_1 \\ v_2
\end{pmatrix}=
\begin{pmatrix}
\boldsymbol{\mathcal{K}} & \boldsymbol{\mathcal{L}}_+ \\ -\boldsymbol{\mathcal{L}}_- & \boldsymbol{\mathcal{K}}
\end{pmatrix}
\begin{pmatrix}
v_1 \\ v_2
\end{pmatrix},
\end{equation}
where 
\begin{gather}
\boldsymbol{\mathcal{L}}_+=\boldsymbol{\mathcal{L}}^\dagger_+=\partial_{xx}+Q_+(x), \qquad
\boldsymbol{\mathcal{L}}_-=\boldsymbol{\mathcal{L}}^\dagger_- =\partial_{xx}+Q_-(x), \label{def:L}
\intertext{and}
\boldsymbol{\mathcal{K}}=-\boldsymbol{\mathcal{K}}^\dagger=\tfrac12R'(x)+R(x)\partial_x. \label{def:K}
\end{gather}
\end{subequations}
Here $Q_\pm(x)$ and $R(x)$ are real-valued, smooth, and $T$-periodic functions for some $T>0$. We introduce 
\begin{align}
w_1= v_1-\tfrac12R(x)v_2\quad\text{and}\quad w_2= -{v_2}_x-\tfrac12R(x)v_1,
\end{align}
which is essentially the nonzero phase analogue of what is presented in \cite{J88}, and we can reformulate \eqref{eqn:L(NLS)} as 
\begin{equation}\label{eqn:A(NLS)}
\begin{pmatrix}
v_1 \\ v_2 \\ w_1 \\ w_2 
\end{pmatrix}_x
=\begin{pmatrix}
0 & \tfrac12R(x) & 1 & 0 \\ 
-\tfrac12R(x) & 0 & 0 & -1 \\ 
-Q_-(x) + \tfrac14R(x)^2 & - \lambda & 0 & -\tfrac12R(x) \\ 
-\lambda & Q_+(x)-\tfrac14R(x)^2 & \tfrac12R(x) & 0 
\end{pmatrix}
\begin{pmatrix}
v_1 \\ v_2 \\ w_1 \\ w_2 
\end{pmatrix}.
\end{equation}
In short, 
\[
\mathbf{v}_x=\mathbf{A}(x,\lambda)\mathbf{v}.
\]
We verify that \text{(A1)} and \text{(A2)} hold for 
\[
\mathbf{B}=-\mathbf{B}^{-1}=\begin{pmatrix} 0 & 0 & -1 & 0 \\ 0 & 0 & 0 & 1 \\ 1 & 0 & 0 & 0 \\ 0 & -1& 0 & 0 \end{pmatrix}.
\]
Additionally, we verify that $\tr(\mathbf{A}(x,\lambda))=0$ for all $x\in\mathbb{R}$ for all $\lambda\in\mathbb{C}$.

We define the monodromy matrix of \eqref{eqn:A(NLS)} as
\[
\mathbf{M}(\lambda)=\mathbf{V}(T,\lambda),
\quad\text{where}\quad \mathbf{V}_x=\mathbf{A}(x,\lambda)\mathbf{V}
\quad\text{and}\quad \mathbf{V}(0,\lambda)=\mathbf{I}_4,
\]
$\mathbf{I}_4$ is the $4\times 4$ identify matrix. We define the characteristic polynomial of the monodromy matrix of \eqref{eqn:A(NLS)} as 
\[
p(\mu,\lambda)=\det(\mathbf{M}(\lambda)-\mu\mathbf{I}_4).
\]
We deduce from Lemma~\ref{lem:symm} that 
\[p(\mu,\lambda)=\mu^4-\boldsymbol{f}(\lambda)\mu^3+\boldsymbol{g}(\lambda)\mu^2-\boldsymbol{f}(-\lambda)\mu+1, \qquad \lambda\in\mathbb{C},
\]
where $\boldsymbol{f}(\lambda)$ and $\boldsymbol{g}(\lambda)$ are defined in \eqref{def:f123}. Importantly, $\boldsymbol{g}(\lambda)$ is real for $\lambda \in i\mathbb{R}$. 

Our interest lies in the $L^2(\mathbb{R})$ essential spectrum of \eqref{eqn:L(NLS)} along the imaginary axis, which holds significant importance when investigating the stability and instability of periodic traveling waves of the nonlinear Schr\"odinger equation and related equations. Theorem~\ref{thm:spec4} establishes distinct regions in $\mathbb{R}^3$ corresponding to varying algebraic multiplicities, by utilizing the Floquet discriminant defined in \eqref{def:f123}. Theorem~\ref{thm:bifur+} demonstrates that if the spectrum of \eqref{eqn:L(NLS)} bifurcates at $\lambda \in i\mathbb{R}$ away from the imaginary axis in a transversal manner then the bifurcation index $\varPhi_4(\lambda)$, defined in \eqref{def:res+}, reaches zero. 

\subsection{Trivial phase solutions}\label{sec:NLS0}

Suppose for simplicity that $\boldsymbol{\mathcal{K}}=0$, and \eqref{eqn:L(NLS)} becomes
\begin{equation}\label{eqn:L(NLS)0}
\begin{aligned} 
&\lambda v_1=\boldsymbol{\mathcal{L}}_+v_2={v_2}_{xx}+Q_+(x)v_2, \\
&\lambda v_2=-\boldsymbol{\mathcal{L}}_-v_1=-{v_1}_{xx}-Q_-(x)v_1
\end{aligned}
\end{equation}
for $\lambda\in\mathbb{C}$, where $Q_\pm(x)$ are real-valued smooth functions satisfying $Q_\pm(x+T)=Q_\pm(x)$ for some $T>0$, the period. This arises in the stability problem for trivial phase solutions of the nonlinear Schr\"odinger equation, among other applications. The first author and Rapti \cite{BR1} have determined the conditions under which the eigenvalues of the monodromy matrix associated with \eqref{eqn:L(NLS)0} lie on the unit circle. It is noteworthy that the monodromy matrix and the characteristic polynomial exhibit additional symmetry. Specifically, \[
\mathbf{M}^\top(x,\lambda)\begin{pmatrix} \mathbf{0} & - \mathbf{I}_2 \\ \mathbf{I}_2 & \mathbf{0}\end{pmatrix}=-\begin{pmatrix} \mathbf{0} & - \mathbf{I}_2 \\ \mathbf{I}_2 & \mathbf{0}\end{pmatrix}\mathbf{M}(x,\lambda)
\]
for all $x\in\mathbb{R}$ for all $\lambda\in\mathbb{C}$. Importantly, the characteristic polynomial of the monodromy matrix takes the form 
\[
p(\mu,\lambda)=\mu^4-f(\lambda)\mu^3+g(\lambda)\mu^2-f(\lambda)\mu+1, 
\]
where both $f(\lambda)$ and $g(\lambda)$ are real-valued and even for $\lambda\in i\mathbb{R}$.

A direct calculation leads to that \eqref{def:pp} becomes
\[
p^\sharp(\nu,\lambda)=(2+2f(\lambda)+g(\lambda))\nu^4+(-12+2g(\lambda))\nu^2+2-2f(\lambda)+g(\lambda).
\]
Let
\begin{align*}
\Delta_4(\lambda):=&\disc_\nu p^\sharp(\nu,\lambda) \\
=&-4096(8+f(\lambda)^2-4g(\lambda))^2(2+2f(\lambda)+g(\lambda))(-2+2f(\lambda)-g(\lambda)),
\end{align*}
and 
\begin{align*}
&P_4(\lambda)=16(-6+g(\lambda))(2-2f(\lambda)+g(\lambda)), \\
&D_4(\lambda)=-256(8+f(\lambda)^2-4g(\lambda))(2-2f(\lambda)+g(\lambda))^2.
\end{align*}
Suppose that $\lambda \in i\mathbb{R}$ and $\Delta_4(\lambda), P_4(\lambda), D_4(\lambda)\neq0$. Following the proof of Theorem~\ref{thm:spec4}, we can establish that the number of distinct eigenvalues of the monodromy matrix for \eqref{eqn:L(NLS)0} on the unit circle or, equivalently, the algebraic multiplicity of the $L^2(\mathbb{R})$ essential spectrum of \eqref{eqn:L(NLS)0} is:
\begin{itemize}
\item four if $\Delta_4(\lambda)>0$ and $P_4(\lambda),D_4(\lambda)<0$, 
\item two if $\Delta_4(\lambda)<0$, 
\item zero if $\Delta_4(\lambda)>0$ and either $P_4(\lambda)$ or $D_4(\lambda)$ is positive.
\end{itemize}

\begin{figure}[htbp]
\centering
\includegraphics[width=0.45\textwidth]{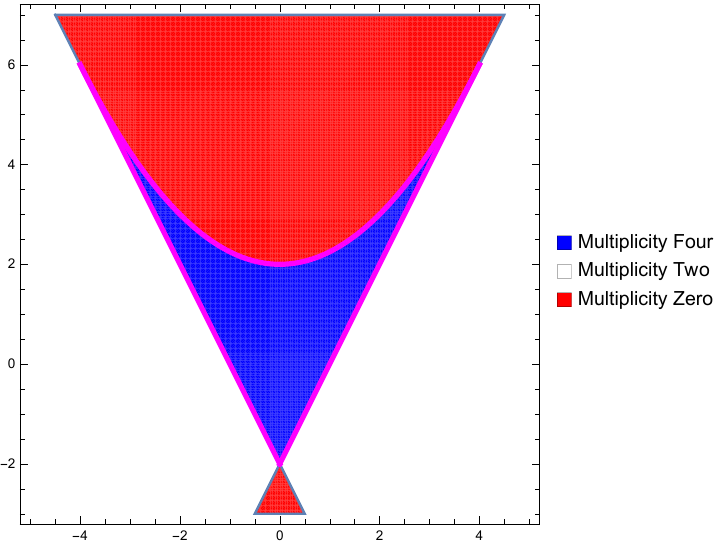}
\caption{The regions in $\mathbb{R}^2$ for different algebraic multiplicities of the $L^2(\mathbb{R})$ essential spectrum of \eqref{eqn:L(NLS)0} along the imaginary axis.} 
\label{fig:region40}
\end{figure}

Figure~\ref{fig:region40} depicts the regions in the $(f,g)$ plane, corresponding to different numbers of eigenvalues of the monodromy matrix for \eqref{eqn:L(NLS)0} on the unit circle or, equivalently, varying algebraic multiplicities of the $L^2(\mathbb{R})$ essential spectrum along the imaginary axis. The blue region corresponds to a multiplicity four, the red region corresponds to a multiplicity zero, and the remaining regions of $\mathbb{R}^2$ corresponds to a multiplicity two. It is important to note that these regions are obtained by restricting the regions in the right panel of Figure~\ref{fig:region34}, in the $(f_1,f_2,f_3)$ coordinates, to the $f_2=0$ plane. 

These regions in Figure~\ref{fig:region40} effectively determine the algebraic multiplicity of the spectrum over a dense open subset of the $(f,g)$ plane. Further investigation would be required to classify the borderline cases where one or more of $\Delta_4, P_4, D_4$ vanishes. Details regarding these cases are not provided here. 

Additionally, let
\begin{align*}
\varPhi_4(\lambda):=&\res_\mu(p(\mu,\lambda),p_\lambda(\mu,\lambda)) \\
=&(g'(\lambda)^2 +(g(\lambda)- 2)f'(\lambda)^2 - f(\lambda) f'(\lambda) g'(\lambda))^2,
\end{align*}
and Theorem~\ref{thm:bifur+} demonstrates that if the $L^2(\mathbb{R})$ essential spectrum of \eqref{eqn:L(NLS)0} bifurcates at $\lambda \in i\mathbb{R}$ away from the imaginary axis then $\varPhi_4(\lambda)=0$.

As an alternative approach, we introduce $\omega = \mu + \frac{1}{\mu}$, which is well-defined because $p(0,\lambda)=1$ for all $\lambda\in\mathbb{C}$. This conformally maps the unit circle to the interval $[-2,2]$. A straightforward calculation leads to that the characteristic equation for the monodromy matrix associated with \eqref{eqn:L(NLS)0} becomes
\[
\omega^2-f(\lambda)\omega+g(\lambda)-2=0,
\]
and the roots are given by
\[
\omega_{\pm}(\lambda)=\frac{f(\lambda)\pm \sqrt{f(\lambda)^2-4g(\lambda)+8}}{2}.
\] 
We can demonstrate that the monodromy matrix for \eqref{eqn:L(NLS)0} has two eigenvalues on the unit circle provided that $\omega_+(\lambda) \in [-2,2]$, and it has additional two eigenvalues on the unit circle provided $\omega_-(\lambda) \in [-2,2]$.
Furthermore, we can show that $\varPhi_4'(\lambda)=0$ if and only if either $\omega_+'(\lambda)=0$ or $\omega_-'(\lambda)=0$. The $L^2(\mathbb{R})$ essential spectrum of \eqref{eqn:L(NLS)0} bifurcates at a critical point of $\omega_\pm$ away from the imaginary axis if and only if $\omega_+(\lambda)$ or $\omega_-(\lambda)$ lies within the interval $(-2,2)$ respectively. The zeros of the bifurcation index do not lead to bifurcation when they are critical points of $\omega_\pm$ but $\omega_\pm(\lambda)$ lies outside the interval $(-2,2)$.

\begin{figure}[htbp]
\centering
\includegraphics[scale=0.75]{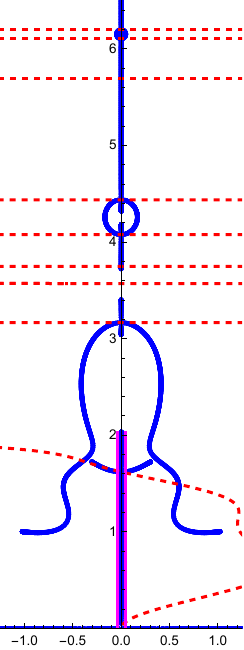}
\includegraphics[scale=0.75]{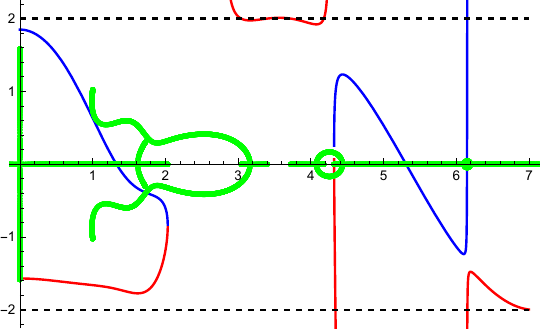}
\caption{On the left, the numerically computed ``frog'' spectrum of \eqref{eqn:eg(NLS)0}, alongside the intervals of algebraic multiplicity four and the graph of the bifurcation index.
On the right, the spectrum, rotated by $90^\circ$, together with the graphs of $\omega_{\pm}$.} 
\label{fig:Floquet4Spectrum}
\end{figure}

For example, Figure~\ref{fig:Floquet4Spectrum} presents the result from a numerical experiment with
\begin{equation}\label{eqn:eg(NLS)0}
\begin{aligned}
&\lambda v_1={v_2}_{xx}+(6-\cos(x)+3\sin(2x))v_2, \\
& \lambda v_2=-{v_1}_{xx}-(4-3\cos(x)+2\sin(2x))v_1. 
\end{aligned}
\end{equation}
In the left panel, the blue curves represent the numerically computed $L^2(\mathbb{R})$ essential spectrum for $2000$ values of the Floquet exponent for each $v_1$ and $v_2$. Similar to Sections~\ref{sec:KdV numerics} and \ref{sec:BBM numerics}, the magenta lines, running parallel to the imaginary axis, indicate intervals of algebraic multiplicity four. The numerical result supports this, revealing that for $\lambda \in i\mathbb{R}$ in the interval $(-\lambda_0 i,\lambda_0 i)$, where $\lambda_0$ is approximately $2$, the algebraic multiplicity is four. Notably, the spectrum includes the interval approximately $(-1.58,1.58)$ in the real axis.

The dashed red curves, representing the graph of the bifurcation index, intersect the imaginary axis, identifying potential bifurcation points where the spectrum might move away from the imaginary axis. It is evident that the bifurcation index $\varPhi_4(\lambda)=0$ at each of these bifurcation points along the imaginary axis. Unlike third-order equations, however, $\varPhi_4(\lambda)=0$ is a necessary condition but not a sufficient condition for such bifurcations. The left panel displays 10 points where the bifurcation index becomes zero, yet only seven of these points result in actual bifurcations, including the interval along the real axis. The three points where the bifurcation index vanishes but does not lead to spectral bifurcations are approximately at $3.6i$, $3.75i$, and $5.7i$. 

The right panel of Figure~\ref{fig:Floquet4Spectrum} shows the spectrum in green, rotated by $90^\circ$, alongside the graphs of $\omega_{\pm}$ in blue and red, respectively. We numerically observe that the algebraic multiplicity of $\lambda \in i\mathbb{R}$ changes when $\omega_\pm(\lambda)$ exit the interval $[-2,2]$, either by remaining real and exiting the interval or by colliding and becoming complex. Additionally, we numerically observe that the critical points of $\omega_{\pm}$ within the interval $(-2,2)$ correspond to spectral bifurcations away from the imaginary axis. Three critical points of $\omega_{\pm}$ occur when $\omega_\pm$ is outside of the interval $(-2,2)$, approximately at $3.6i$, $3.75i$, and $5.7i$, but they do not lead to bifurcations. The critical point of $\omega_-$ at around $3.6i$ is visible, while the other two critical points outside of $(-2,2)$ are not visible at this scale.

\subsection{Nontrivial phase solutions}\label{sec:NLSn0}

We turn our attention to the nonlinear Schr\"odinger equation
\begin{equation}\label{eqn:NLS}
i \psi_t = \psi_{xx} + g(|\psi|^2)\psi
\end{equation}
for some nonlinearity $g$, and nontrivial phase solutions of the form 
\[
\psi(x,t)= A(x)e^{iK(x)+i\omega t}
\quad \text{for some functions $A(x)$ and $K(x)$}.
\]
As is the case for the KdV equation the stability of the traveling wave solutions to the NLS equation has been extensively studied: see \cite{BuslaevSulem03, CazenaveLions82,GallayHaragus07,GallayHaragus2007,Weinstein1985} for some examples.

The spectral problem for \eqref{eqn:NLS} concerning a nontrivial phase solution does not possess the additional symmetry present in the spectral problem for a trivial phase solution. Consequently, the trace of the monodromy matrix $\tr(\mathbf{M}(\lambda))$ takes complex values. However, it is noteworthy that $\frac{1}{2}(\tr(\mathbf{M}(\lambda))^2-\tr(\mathbf{M}^2(\lambda)))$ is necessarily real for $\lambda \in i\mathbb{R}$. 

In what follows, we select $g(x)=3x^2$---namely, the quintic focusing nonlinearity. 
A direct calculation leads to
\[
A_{xx}-\frac{\kappa^2}{A^3}+\omega A+3A^5=0
\quad\text{and}\quad K_x=\frac{\kappa}{A^2}.
\]
Here we select $\kappa=\frac{3\sqrt{2}}{8}$ and $\omega = \frac{31}{16}$. After integration, we obtain
\[
A_x^2=-\frac{21}{32}-\frac{9}{32 A^2}+\frac{31}{16}A^2-A^6.
\]
The turning points, which are the zeros of the right side, include $A=\pm1$, $\pm\frac{i}{2}$, $\pm\frac{\sqrt{6}}{2}$, and $\pm i\frac{\sqrt{6}}{2}$, with a period approximately $1.93$.

\begin{figure}[htbp]
\centering
\includegraphics[width=0.45\textwidth]{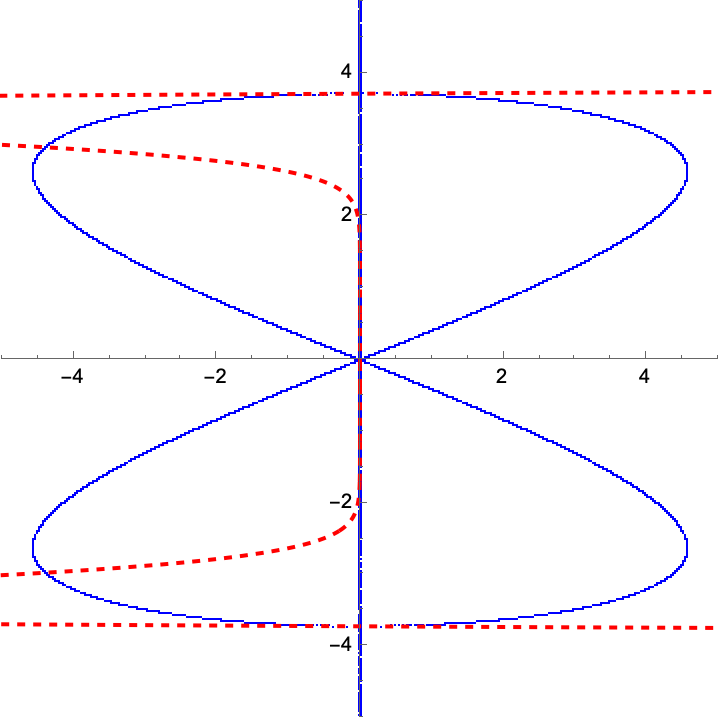}
\caption{The numerically computed spectrum of the linearized operator of the focusing quintic Schr\"odinger equation about a nontrivial phase solution.} 
\label{fig:NLS_Phase}
\end{figure}

Figure~\ref{fig:NLS_Phase} depicts the numerically computed $L^2(\mathbb{R})$ essential spectrum of the linearization of \eqref{eqn:NLS} with the quintic focusing nonlinearity about a nontrivial phase solution. The numerical result shows that in the depicted range, the imaginary axis is included in the spectrum with an algebraic multiplicity two. The dashed red curves represent the graph of the bifurcation index, and their intersections with the imaginary axis identify potential bifurcation points. The numerical result confirms that the spectrum bifurcates from the imaginary axis at $0$ and approximately $\pm 3.72 i$. 
 
\subsection{The Boussinesq equation}\label{sec:Boussinseq}

Another illustrative example is the Boussinesq equation
\begin{equation}\label{eqn:bous}
u_{tt} -u_{xx}+u_{xxxx}-g(u)_{xx} = 0
\end{equation}
for some appropriate nonlinearity $g$ or, alternatively,
\begin{equation} \label{eqn:bous'}
u_{tt}-2cu_{xt}+(c^2-1)u_{xx}+u_{xxxx}-g(u)_{xx}=0
\end{equation}
in the frame of reference moving at some $c\neq0, \in \mathbb{R}$, the wave speed. Typically, the nonlinearity $g(u)$ is chosen to be $u^2$, but other choices are possible. Let $\phi(x)$ denote a standing wave solution of \eqref{eqn:bous'} or, equivalently, $\phi(x-ct)$ represents a traveling wave solution of \eqref{eqn:bous}. We assume that $\phi(x+T)=\phi(x)$ for some $T>0$, the period.

Linearizing \eqref{eqn:bous'} about a standing wave solution $\phi(x)$ leads to\cite{Quintero03,stefanov_2014}
\begin{equation}\label{eqn:L(bous)}
\lambda^2 v-2c\lambda v_x+v_{xxxx}+(c^2-1)v_{xx}-(g(\phi)v)_{xx}=0,\qquad \lambda\in\mathbb{C}.
\end{equation} 
Introducing
\begin{align*}
&w=v_{xx}-(g(\phi)+1 -c^2)v:=v_{xx}-Q(x)v, \\ 
&v_1=v_x \quad\text{and}\quad -\lambda^2 w_1=w_x,
\end{align*}
and we can reformulate \eqref{eqn:L(bous)} as 
\begin{equation}\label{eqn:A(bous)}
\begin{pmatrix}
v \\ w \\ v_1 \\ w_1 
\end{pmatrix}_x
=\begin{pmatrix}
0 & 0 & 1 & 0 \\ 0 & 0 & 0 & -\lambda^2 \\ Q(x) & 1 & 0 & 0 \\ 1 & 0 & -\tfrac{2c}{\lambda} & 0 
\end{pmatrix}
\begin{pmatrix}
v \\ w \\ v_1 \\ w_1 
\end{pmatrix}.
\end{equation}
In short, 
\[
\mathbf{v}_x=\mathbf{A}(x,\lambda)\mathbf{v}.
\]
We verify that \text{(A1)} and \text{(A2)} hold for
\[
\mathbf{B}(\lambda)=\begin{pmatrix}
8 c^3 & - 2c & \lambda & - 4 c^2 \lambda \\ 
-2 c & 0 & 0 & \lambda \\ 
-\lambda & 0 & 0 & 0 \\ 
4 c^2 \lambda & - \lambda & 0 & 2c \lambda^2
\end{pmatrix},
\quad \text{where} \quad
\mathbf{B}(\lambda)^{-1} = 
\begin{pmatrix}
0 & 0 & -\tfrac1\lambda & 0 \\ 
0 & -2c & 0 & -\tfrac1\lambda \\ 
\tfrac1\lambda & 0 & 0 &-\tfrac{2c}{\lambda^2} \\ 
0 & \tfrac1\lambda & -\tfrac{2c}{\lambda^2} & 0 
\end{pmatrix}.
\]
Note that \eqref{eqn:A(bous)} is infinitesimally symplectic when $c =0$. Indeed, $\mathbf{A}(x,-\lambda)=\mathbf{A}(x,\lambda)$ and $\mathbf{B}(\lambda)=\begin{pmatrix} \mathbf{0} & \lambda\mathbf{I}_2 \\ -\lambda\mathbf{I}_2 & \mathbf{0}\end{pmatrix}$ when $c=0$. However, rather than having additional symmetry as seen in the case of the nonlinear Schr\"odinger equation, it appears that the symmetry reduces to the symplectic one. 

\begin{figure}[htbp]
\centering
\includegraphics[width=0.45\textwidth]{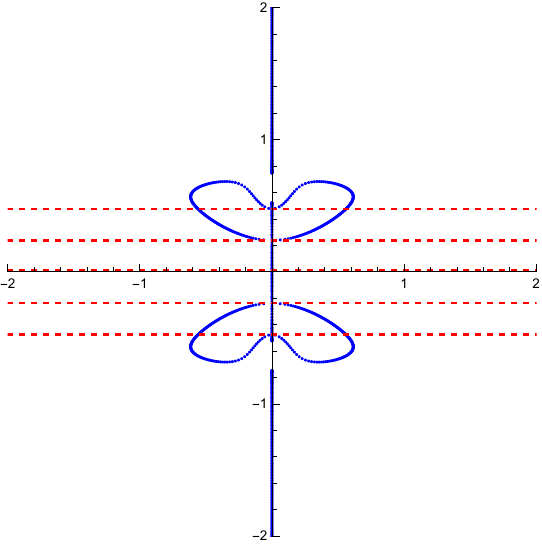}
\caption{The numerically computed ``maple seed'' spectrum for \eqref{eqn:A(bous)} for \eqref{def:Q(bous)}. The spectrum is depicted in blue, and the bifurcation index in red dashed.} 
\label{fig:Boussinesq4Wingnut}
\end{figure}

Figure~\ref{fig:Boussinesq4Wingnut} presents the result from a numerical experiment with 
\begin{equation}\label{def:Q(bous)}
Q(x)= 5\cos(x)+\sin(2x)\quad\text{and}\quad c=1.
\end{equation}
We pause to remark that this potential function does not arise in the stability problem for periodic traveling waves of the Boussinesq equation, as far as we are aware. The blue curves are the numerically computed $L^2(\mathbb{R})$ essential spectrum of \eqref{eqn:A(bous)} for \eqref{def:Q(bous)}. We numerically observe that within the depicted range, the imaginary axis is included in the spectrum, and the algebraic multiplicity is zero or two. In fact, no portion of the spectrum has an algebraic multiplicity four. The dashed red lines represent the graph of the bifurcation index, and their intersections with the imaginary axis indicate potential bifurcation of the spectrum away from the imaginary axis.

\begin{figure}[htbp]
\centering
\includegraphics[width=0.45\textwidth]{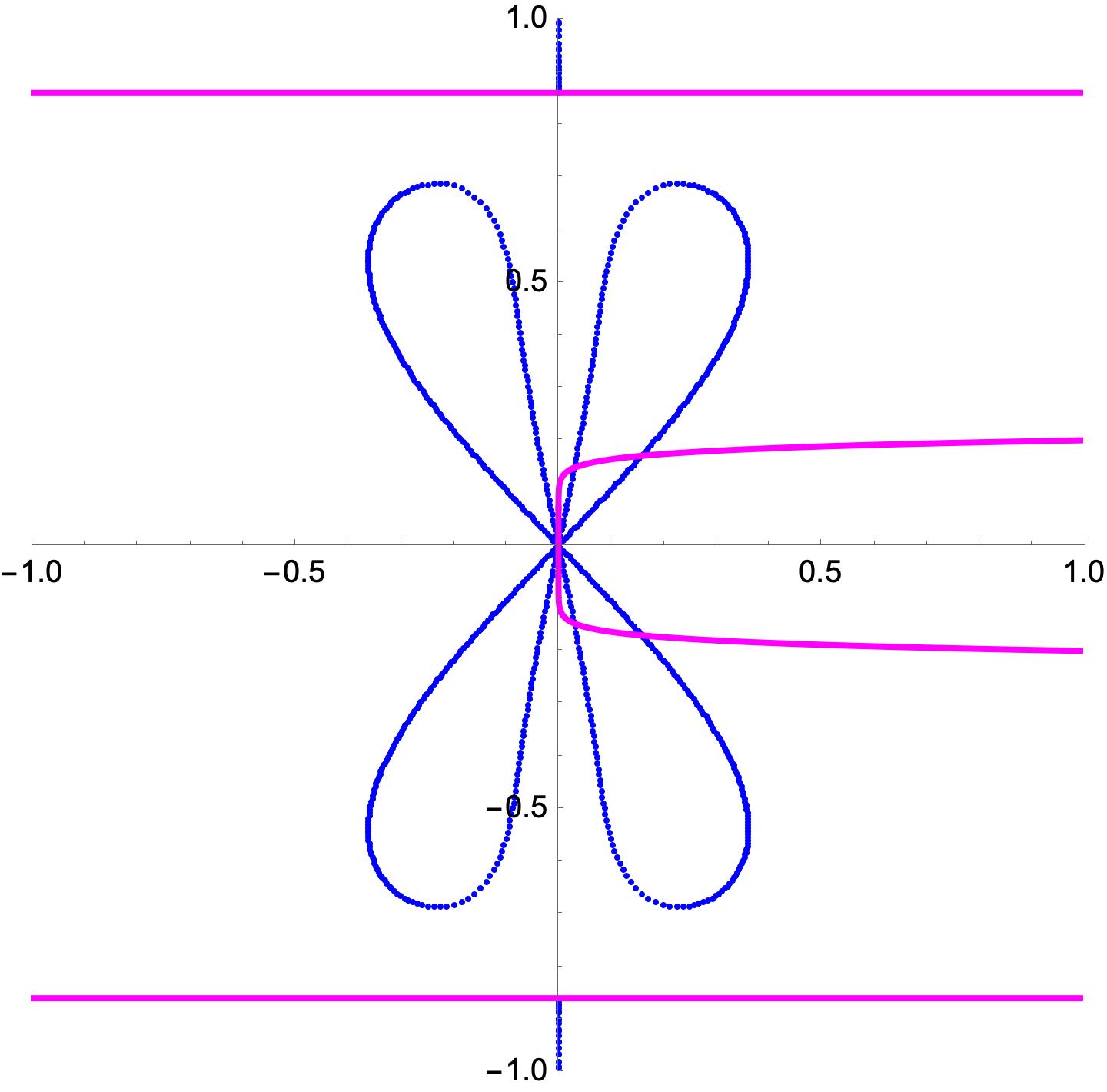}
\caption{The numerically computed spectrum of the linearized operator of \eqref{eqn:bous'2} about \eqref{eqn:phi(bous)}. The magenta curves represent the sign of the discriminant of the characteristic polynomial.} 
\label{fig:BoussinesqMultiplicityZero}
\end{figure}

Let's consider the Boussinesq equation with a cubic power-law nonlinearity
\begin{equation}\label{eqn:bous'2}
u_{tt}-u_{xx}+u_{xxxx}-(u^3-u)_{xx}=0
\end{equation}
and a periodic traveling wave solution in closed form
\begin{equation}\label{eqn:phi(bous)}
u(x,t)=\tfrac{\sqrt{7}}{2}\operatorname{dn}\Big(\sqrt{\tfrac78}(x-t),\tfrac67\Big).
\end{equation}
Here ``$\operatorname{dn}$'' represents a Jacobi elliptic function, the delta amplitude. Note that \eqref{eqn:phi(bous)} oscillates between a minimum of $\tfrac12$ and a maximum of $\tfrac{\sqrt{7}}2$ with a period $2 \sqrt{\tfrac87}\operatorname{K}(\tfrac67)\approx 5.156$, where $\operatorname{K}$ is the complete elliptic integral of the first kind. 

Figure~\ref{fig:BoussinesqMultiplicityZero} shows the numerically computed $L^2(\mathbb{R})$ essential spectrum of the linearization operator of \eqref{eqn:bous'2} about \eqref{eqn:phi(bous)}. The magenta curves represent the sign of the discriminant of the characteristic polynomial. The numerical result reveals that the algebraic multiplicity is either zero or four when the discriminant is positive, while the multiplicity is two when the discriminant is negative. The transition between multiplicities zero and two is detected by the change in the sign of the discriminant. There do not appear to be any intervals where the multiplicity is four, and the bifurcation index seems to have a simple zero at $0\in\mathbb{C}$. 

\subsection{Spectrum along the imaginary axis towards $\pm i\infty$}\label{sec:asym4}

When applying \eqref{eqn:res(infty)} to a fourth-order equation, we obtain 
\[
\res_4(\lambda)\sim\frac{T^4\sin^2(\sqrt[4]{\lambda}T)\sinh^2(\sqrt[4]{\lambda}T)(\cos\small(\sqrt[4]{\lambda}T\small)-\cosh\small(\sqrt[4]{\lambda}T\small))^4}{\lambda^3},
\]
which is generally not real even if $\lambda$ belongs to the imaginary axis. This reflects the fact that the limiting eigenvalue problem, which is the ODE with $\lambda$ as a parameter solved by the leading-order WKB approximations, is generally not Hamiltonian for a spectral problem of even order. Consequently, we do not expect to find information about spectrum and, hence, bifurcations along the imaginary axis. However, for the even-order equations discussed here, such as the Boussinesq equation and the nonlinear Schr\"odinger equation both with and without a phase, the WKB approximation gives a leading-order approximation in the spectral parameter that corresponds to the solutions of a Hamiltonian ODE. Therefore, in theory, we can produce leading-order approximations of the discriminant of the characteristic polynomial as well as the bifurcation index for such equations.

For the case of the nonlinear Schr\"odinger equation, both with and without a phase, we set $\lambda=i\nu^2$ for $\nu \gg 1$. Plugging this into the formula, we obtain the following leading-order expression for the discriminant of the characteristic polynomial:
\[
\Delta_4(i\nu^2) \sim -256\sin^2(\nu T) \sinh^2(\nu T)(\cos(\nu T)-\cosh(\nu T))^4 \leq0.
\]
Consequently, there are two eigenvalues of the monodromy matrix on the unit circle outside the discrete set of $\mathbb{C}$ for $\lambda \in i\mathbb{R}$ and $|\lambda|\gg1$. Furthermore, the leading-order expression for the bifurcation index is
\[
\varPhi_4(i\nu^2)\sim \frac{16T^4\sin^2(\nu T)\sinh^2(\nu T)(\cos(\nu T)-\cosh (\nu T))^4}{\nu^4}\geq 0.
\]
However, this vanishes to an even order when $\nu$ is an integer multiple of $\tfrac{2 \pi}{T}$, suggesting that bifurcations of isolas are possible near these points, depending on how the bifurcation index perturbs. Determining whether there is an infinite number of bifurcations along the imaginary axis would require higher-order WKB approximations. For the nonlinear Schr\"odinger equation, with or without a phase, the next-order non-vanishing WKB approximation depends on the potential, and it can be numerically determined in any particular case. However, to the best of our knowledge, it is not particularly amenable to general analysis. 

For the Boussinesq equation, the leading-order discriminant and bifurcation index are the same as those for the nonlinear Schr\"odinger equation. Therefore, higher-order WKB approximations, which depend on the potential, would be required to analyze the possibility of an infinite number of bifurcations.

\section{The Kawahara equation}\label{sec:5}

We conclude by considering a fifth-order KdV equation known as the Kawahara equation
\begin{equation}\label{eqn:Kawahara}
u_t-u_{xxxxx}+\alpha u_{xxx}+uu_x= 0,\qquad \alpha \in \mathbb{R}.
\end{equation}
As with previous equations there have been a number of papers that have considered the stability of traveling wave solutions to the Kawahara equation\cite{BuryakChampneys97,HaragusLombardiScheel06,levandosky07,Natali10,TanYangPelinovsky02,TrichtchenkoDeconinckKollar18}. 

We will begin by providing comprehensive discussion of periodic traveling waves of  \eqref{eqn:Kawahara}, given explicitly in terms of a Jacobi elliptic function. This expands on the treatment in \cite{KanoNakayama}. In the case of $\alpha=0$, \eqref{eqn:Kawahara} admits a one-parameter family of periodic stationary wave solutions given explicitly as
\begin{equation}\label{eqn:KdV5(0)}
\phi(x)= 420\sigma^4\operatorname{cn}^4\big(\sigma x,\tfrac12\big)-168\sigma^4, 
\qquad \sigma \in \mathbb{R}.
\end{equation}
Since \eqref{eqn:Kawahara} enjoys the same Galilean invariance as the generalized KdV equation, we can boost and translate the solution to obtain periodic traveling wave solutions. However, since these transformations do not alter the stability or instability of the solution, we will limit our discussion to the form \eqref{eqn:KdV5(0)}. 

For $\alpha\neq0$, periodic traveling wave solutions of \eqref{eqn:Kawahara} become more intricate due to the loss of scaling invariance present in the $\alpha=0$ case. Nevertheless, we obtain a one-parameter family of periodic stationary wave solutions, represented as
\begin{subequations}\label{eqn:Kawahara0}
\begin{equation}
\phi(x)=C_1+C_2\operatorname{cn}^2(\sigma x,m^*)+C_3\operatorname{cn}^4(\sigma x,m^*), 
\end{equation}
where
\begin{align}
&C_1=\frac{-31\alpha^2+264992\sigma^4{m^*}^2-264992\sigma^4 {m^*}-18928\sigma^4+3640 \alpha\sigma^2-7280\alpha\sigma^2 {m^*}}{507}, \notag\\
&C_2=-\frac{280}{13} \sigma^2 m^*(-\alpha +104\sigma^2 m^*-52\sigma^2), \\
&C_3=1680 \sigma^4 {m^*}^2. \notag
\end{align}
\end{subequations}
This family of traveling wave solutions is parameterized by $\sigma$ and, additionally, by Galilean and translational invariance, which we do not include here. It is worth noting that, unlike the $\alpha=0$ case, where $\sigma$ can take any real value, for $\alpha\neq 0$, $\sigma$ must satisfy $|\frac{\alpha}{\sigma^2}|<52$. Interestingly, the number $52$ is not a numerical approximation but the exact value, corresponding to the unique real root of 
\[
31x^3-56784x-1406080=0.
\] 
When $|\frac{\alpha}{\sigma^2}|<52$, \eqref{eqn:Kawahara0} defines periodic traveling waves of \eqref{eqn:Kawahara} and they are characterized by an elliptic parameter $m^*$ determined by the unique root of 
\begin{equation}\label{eqn:mPoly}
-703040\,(m-2)(m+1)(2 m-1)+56784\frac{\alpha}{\sigma^2}(m^2-m+1)-31\left(\frac{\alpha}{\sigma^2}\right)^3=0
\end{equation}
within the interval $m\in(0,1)$. It is important to note that changing the sign of $\alpha$ results in exchanging the parameter $m$ with its complementary parameter $m'=1-m$. This arises from the imaginary Jacobi transformation, which relates elliptic functions with argument $x$ and parameter $m$ to those with argument $ix$ and parameter $1-m$, together with the fact that under a complex rotation $x\mapsto ix$, we have $ \partial_{xx} \mapsto -\partial_{xx}$ and $ \partial_{xxxx} \mapsto \partial_{xxxx}$.

\begin{figure}[htbp]
\centering
\includegraphics[width=0.45\textwidth]{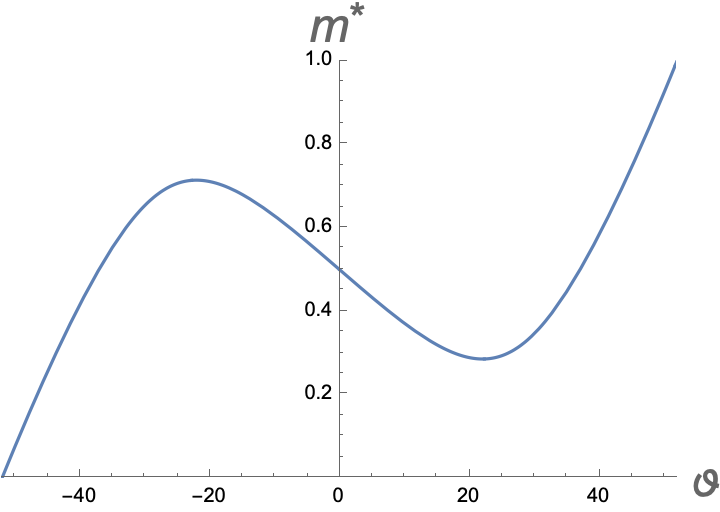}
\caption{$m^*$ as a function of $\frac{\alpha}{\sigma^2}$.}\label{fig:Kawam}
 \end{figure}

Figure~\ref{fig:Kawam} depicts the elliptic parameter $m^*$ as a function of $\frac{\alpha}{\sigma^2}$ within the range $(-52,52)$. The curve exhibits critical points $m^*\approx 0.285665$ and $m^*\approx 0.714335$, corresponding to the two roots of the polynomial 
\[
64 - 192 m - 391 m^2 + 1102 m^3 - 391 m^4 - 192 m^5 + 64 m^6,
\]
which is the discriminant of \eqref{eqn:mPoly} with respect to $\frac{\alpha}{\sigma^2}$. These critical points reside in the interval $(0,1)$. 

Additionally, we will rely on results regarding the real roots of a real quintic polynomial, following the notation and methodology in \cite{Chaundy}. To summarize, a real quintic polynomial
\[
p(x)=a_0x^5+5a_1x^4+10a_2x^3+10a_3x^2+5a_4x+a_5,
\]
$a_k \in \mathbb{R}$, can be transformed into a monic depressed quintic polynomial
\[
p_0(x)=x^5+10A_2x^3+10A_3x^2+5A_4z+A_5,
\]
using the change of variables $x\mapsto \frac{x-a_1}{a_0}$. Let
\begin{align*}
&\Delta_5=\disc(p_0), \\
&P_5=4A_2^3+A_3^2, \\
&D_5=A_5^2+16A_2A_4^2-76A_2A_3A_5-(272A_2^3-108A_3^2)A_4+24A_2^2(40A_2^3+27A_3^2). 
\end{align*}
Suppose $\Delta_5,P_5,D_5\neq0$. Similar to quartic polynomials, the roots of $p_0$ and, hence, $p$ have:
\begin{itemize}
\item five real roots if $\Delta_5>0$ and $P_5<0$, $D_5<0$, 
\item three real roots if $\Delta_5<0$,
\item one real root if $\Delta_5>0$ and either $P_5>0$ or $D_5>0$.
\end{itemize}
For simplicity, we omit discussion of various higher co-dimension possibilities, where one or more of $\Delta_5,P_5,D_5$ vanishes. It is worth noting that these quantities can be expressed entirely in terms of the Floquet discriminant for \eqref{eqn:Kawahara}, but the resulting analytical expressions are quite extensive, and we do not reproduce them there.  

\subsection{Numerical experiments}\label{sec:Kawahara numerics}

\begin{figure}[htbp]
\centering
\includegraphics[scale=0.35]{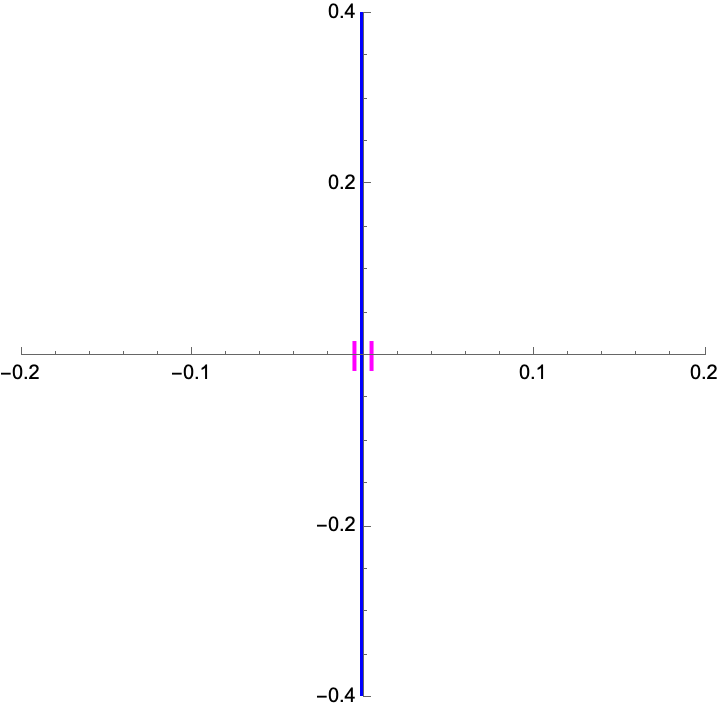}
\includegraphics[scale=0.75]{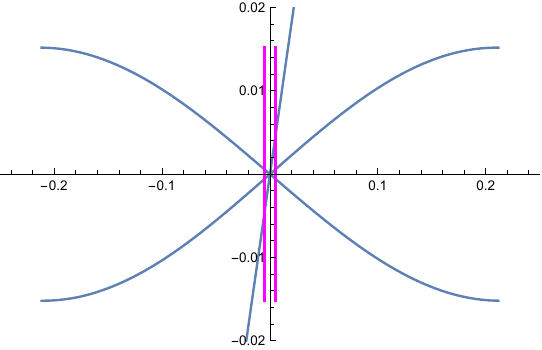}
\caption{(Left) The numerically computed spectrum of the linearization of \eqref{eqn:Kawahara} about \eqref{eqn:KdV5(0)} for $\alpha=0$ and $\sigma=\frac14$. (Right) A close-up near the interval of multiplicity five, alongside the graphs of the triple covering $\lambda_{-1}$, $\lambda_0$ and $\lambda_1$.}   
\label{fig:KawaharaSpectrum}
\end{figure}

We begin our numerical experiment by setting $\alpha=0$ and focusing on the periodic standing wave solutions in \eqref{eqn:KdV5(0)}. Note that neither the scaling invariance transformation $u(x,t) \mapsto \sigma^4 u(\sigma x, \sigma^9 t)$ nor the Galilean invariance transformation $u(x,t) \mapsto c + u(x-ct,t)$ affects the stability or instability of the underlying wave. Consequently, the results obtained in our numerics apply to the entire family of traveling waves under consideration. 

For our numerical computation, we have chosen $\sigma=\frac{1}{4}$. From general modulation theoretic considerations, we expect that at $0\in\mathbb{C}$, the associated monodromy matrix should have an eigenvalue $\mu=1$ with an algebraic multiplicity three. Our numerical findings indeed confirms this and the spectral problem is numerically stiff. Specifically, at $\lambda=0$, the eigenvalues of the monodromy matrix include $\mu=1$ with a multiplicity three, $\mu\approx 28284.5$ and $\mu \approx 3.536\times 10^{-5}.$

In Figure \ref{fig:KawaharaSpectrum}, the blue curves represent the numerically computed $L^2(\mathbb{R})$ essential spectrum of the linearized operator of \eqref{eqn:Kawahara}, without third-order dispersion, about \eqref{eqn:KdV5(0)} for $\sigma=\frac14$. On the left, the spectrum is observed to be entirely aligned with the imaginary axis, and all numerically computed eigenvalues have a maximum real part approximately $3.4\times10^{-10}$. The magenta lines, running parallel to the imaginary axis, indicate intervals where the spectrum has an algebraic multiplicity three. This interval is approximately $(-0.015 i, 0.015 i)$. To validate the determination of the region with multiplicity three, we compared it with the results obtained from the spectral computation. As the Floquet exponent $\mu$ varies over a range, it becomes evident that the interval $(-0.015 i, 0.015 i)$ is triply covered. The right panel shows the imaginary parts of the triple covering $\lambda_{-1}$, $\lambda_0$, and $\lambda_1$ as functions of $\mu$. The independently computed region of multiplicity three, utilizing the Floquet discriminant, is in magenta, and the two are in excellent agreement.

\begin{figure}[htbp]
\centering
\includegraphics[scale=0.35]{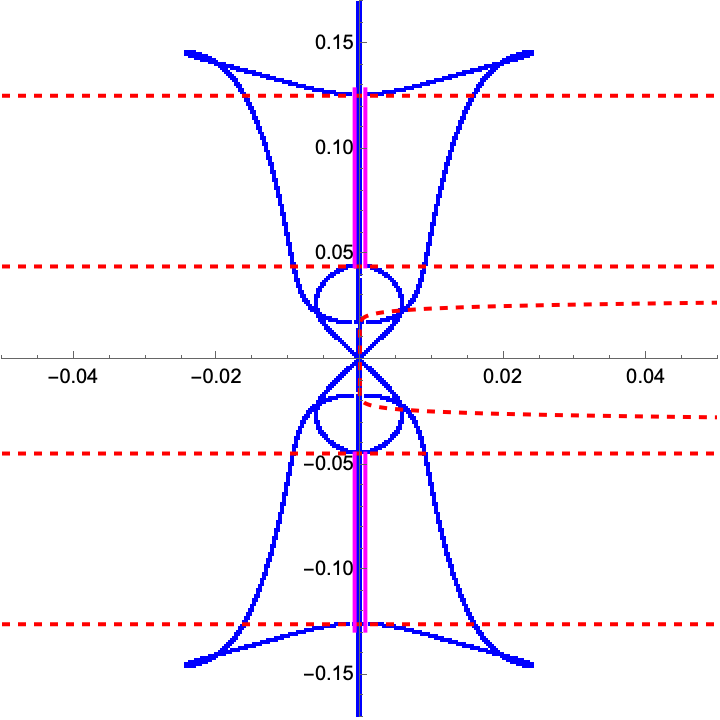}
\includegraphics[scale=0.45]{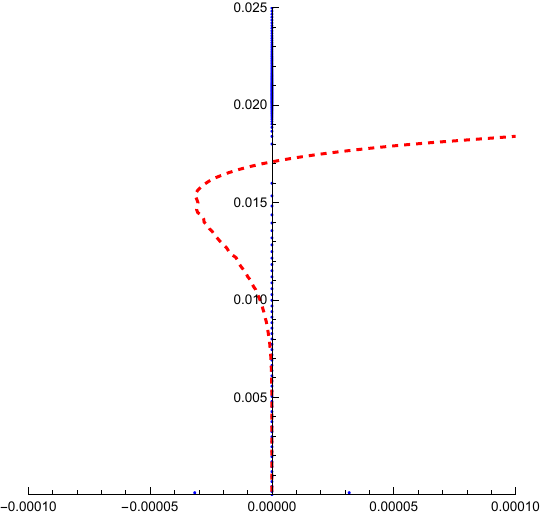}
\includegraphics[scale=0.45]{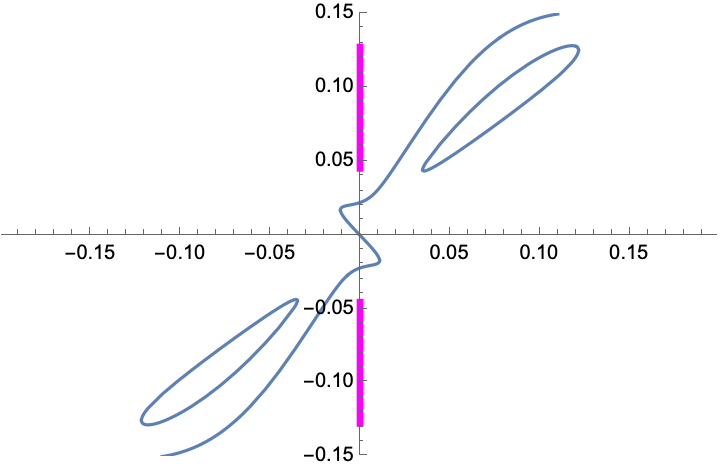}
\caption{(Left) The numerically computed spectrum of the linearization of \eqref{eqn:Kawahara} about \eqref{eqn:Kawahara0} for $\alpha=-2$ and $\sigma=\frac14$. (Middle) A close-up of the bifurcation index near $0\in\mathbb{C}$. (Right) The graphs of triple covering $\lambda_0$ and $\lambda_{\pm1}$ as functions of $\mu$. The magenta lines are the intervals in $i\mathbb{R}$ of multiplicity three. }   
\label{fig:KawaharaSpectrum2}
\end{figure}

Let's examine the spectral problem for \eqref{eqn:Kawahara} and the periodic traveling wave solution in \eqref{eqn:Kawahara0} for $\alpha=-2$ and $\sigma=\frac14$. This gives an elliptic parameter $m \approx .6185$. It is important to note that for this particular sign of $\alpha$, the two kinetic energy terms in the Hamiltonian, $\int u_x^2~dx$ and $\int u_{xx}^2~dx$, exhibit opposite signs. Based on physical intuition, the kind of competition between these two terms might lead to instabilities. Our numerical results indeed support the intuition. 

In the left panel of Figure~\ref{fig:KawaharaSpectrum2}, the blue curves represent the numerically computed $L^2(\mathbb{R})$ essential spectrum of the linearization of \eqref{eqn:Kawahara} about \eqref{eqn:Kawahara0} for $\alpha=-2$ and $\sigma=\frac14$. We numerically observe that the spectrum bifurcates from the imaginary axis at the origin and at six other points along the imaginary axis. Notably, the dashed red curves intersect the imaginary axis precisely at these bifurcation points. Within an interval along the imaginary axis around $0\in\mathbb{C}$, excluding the origin itself, the algebraic multiplicity is one. On the other hand, there exist two short intervals, in magenta, on the imaginary axis with a multiplicity three, away from $0\in\mathbb{C}$.  

On the right, the blue curves show the imaginary parts of three purely imaginary eigenvalues as functions of the Floquet exponent $\mu$, accompanied by subintervals on the imaginary axis where the multiplicity is three, in magenta. It is evident that the two are in very close agreement. It is worth emphasizing that these graphs were computed using different computer codes. The magenta intervals were determined using an ODE solver along the imaginary axis to compute the Floquet discriminant, while the blue curves were obtained through spectral decomposition together with an eigenvalue solver.

\subsection{Spectrum along the imaginary axis towards $\pm i\infty$}\label{sec:asym5}

Similar to the case of the generalized KdV equation, Theorem~\ref{thm:asym} applies to the Kawahara equation. Specifically, for $\lambda=i\nu^5$ and $\nu \gg 1$, the discriminant of the characteristic polynomial is approximated as
\begin{align*}
\Delta_5(i\nu^5)\sim&4096\left(\cos(\tfrac{1}{2}\sqrt{5}\nu T)
-\cosh(\tfrac{1}{2}\sqrt{5-2\sqrt{5}}\nu T)\right)^2 \\ 
& \times \left(\cos(\tfrac{1}{4}(5+\sqrt{5}) \nu T)
-\cosh(\tfrac{1}{2}\sqrt{\tfrac{1}{2}(5-\sqrt{5})} \nu T)\right)^2 \\ 
& \times \left(\cos(\tfrac{1}{4}(\sqrt{5}-5) \nu T)
-\cosh(\tfrac{1}{2}\sqrt{\tfrac{1}{2}(5+\sqrt{5})} \nu T)\right)^2 \\ 
& \times \left(\cos(\tfrac{1}{2}\sqrt{5} \nu T)
-\cosh(\tfrac{1}{2}\sqrt{5+2\sqrt{5}} \nu T)\right)^2 \\ 
& \times \sinh ^2\left(\tfrac{1}{2}\sqrt{\tfrac{1}{2}(5-\sqrt{5})} \nu T\right) 
\sinh^2\left(\tfrac{1}{2}\sqrt{\tfrac{1}{2}(5+\sqrt{5})} \nu T\right).
\end{align*}
While this expression may appear somewhat intricate, it is real and positive for all real values of $\nu$ for all $T>0$. Subsequently, \eqref{eqn:res(infty)} provides the formula for the bifurcation index. Our verification confirms, as with the generalized KdV equation, that only a finite number of isolas bifurcate away from the imaginary axis for the Kawahara equation.

\bibliographystyle{amsplain} 
\bibliography{HamiltonianFloquet}

\providecommand{\bysame}{\leavevmode\hbox to3em{\hrulefill}\thinspace}
\providecommand{\MR}{\relax\ifhmode\unskip\space\fi MR }
\providecommand{\MRhref}[2]{%
  \href{http://www.ams.org/mathscinet-getitem?mr=#1}{#2}
}
\providecommand{\href}[2]{#2}
\begin{thebibliography}{10}

\bibitem{BH}
Jared~C Bronski and Vera~Mikyoung Hur, \emph{Modulational instability and
  variational structure}, Studies in Applied Mathematics \textbf{132} (2014),
  no.~4, 285--331.

\bibitem{BHS2023}
Jared~C. Bronski, Vera~Mikyoung Hur, and Sarah Simpson, \emph{Bounds on
  unstable eigenvalues for stability problems of kdv type.}, in preparation
  (2023).

\bibitem{BJ}
Jared~C. Bronski and Mathew~A. Johnson, \emph{The modulational instability for
  a generalized {K}orteweg-de {V}ries equation}, Arch. Ration. Mech. Anal.
  \textbf{197} (2010), no.~2, 357--400. \MR{2660515}

\bibitem{BJK1}
Jared~C Bronski, Mathew~A Johnson, and Todd Kapitula, \emph{An index theorem
  for the stability of periodic travelling waves of {K}orteweg--de {V}ries
  type}, Proceedings of the Royal Society of Edinburgh Section A: Mathematics
  \textbf{141} (2011), no.~6, 1141--1173.

\bibitem{BR1}
Jared~C. Bronski and Zoi Rapti, \emph{Modulational instability for nonlinear
  {S}chr\"{o}dinger equations with a periodic potential}, Dyn. Partial Differ.
  Equ. \textbf{2} (2005), no.~4, 335--355. \MR{2193631}

\bibitem{BuryakChampneys97}
A.V. Buryak and A.R. Champneys, \emph{On the stability of solitary wave
  solutions of the fifth-order {KdV} equation}, Physics Letters A \textbf{233}
  (1997), no.~1, 58--62.

\bibitem{BuslaevSulem03}
Vladimir~S Buslaev and Catherine Sulem, \emph{On asymptotic stability of
  solitary waves for nonlinear {S}chr{\"o}dinger equations}, Annales de
  l'Institut Henri Poincar{\'e} C, Analyse non lin{\'e}aire, vol.~20, Elsevier,
  2003, pp.~419--475.

\bibitem{CazenaveLions82}
T.~Cazenave and P.~L. Lions, \emph{Orbital stability of standing waves for some
  nonlinear {S}chr{\"o}dinger equations}, Communications in Mathematical
  Physics \textbf{85} (1982), no.~4, 549--561.

\bibitem{Chaundy}
TW~Chaundy, \emph{On the number of real roots of a quintic equation}, The
  Quarterly Journal of Mathematics (1934), no.~1, 10--22.

\bibitem{DeconinckKapitula10}
Bernard Deconinck and Todd Kapitula, \emph{The orbital stability of the cnoidal
  waves of the {K}orteweg--de {V}ries equation}, Physics Letters A \textbf{374}
  (2010), no.~39, 4018--4022.

\bibitem{DeconinckKutz_2006}
Bernard Deconinck and J.~Nathan~Kutz, \emph{Computing spectra of linear
  operators using the floquet-fourier-hill method}, J. Comput. Phys.
  \textbf{219} (2006), no.~1, 296–321.

\bibitem{Eastham1973}
MSP Eastham, \emph{The spectral theory of periodic differential equations},
  Scottish Academic Press, Edinburgh and London, 1973.

\bibitem{GallayHaragus07}
Thierry Gallay and Mariana H{\u{a}}r{\u{a}}gu{\c{s}}, \emph{Stability of small
  periodic waves for the nonlinear {S}chr{\"o}dinger equation}, Journal of
  Differential Equations \textbf{234} (2007), no.~2, 544--581.

\bibitem{GallayHaragus2007}
Thierry Gallay and Mariana Hǎrǎgus, \emph{Orbital stability of periodic waves
  for the nonlinear {S}chr{\"o}dinger equation}, Journal of Dynamics and
  Differential Equations \textbf{19} (2007), 825--865.

\bibitem{GKZ}
Israel~M Gelfand, Mikhail~M. Kapranov, and Andrei~V. Zelevinsky,
  \emph{Discriminants, resultants, and multidimensional determinants.
  {M}athematics: Theory \& applications}, Birkh{\"a}user Boston Inc., Boston,
  MA, 1994.

\bibitem{stefanov_2014}
Sevdzhan Hakkaev, Milena Stanislavova, and Atanas Stefanov, \emph{Linear
  stability analysis for periodic travelling waves of the boussinesq equation
  and the klein–gordon–zakharov system}, Proceedings of the Royal Society
  of Edinburgh Section A: Mathematics \textbf{144} (2014), no.~3, 455–489.

\bibitem{HaragusLombardiScheel06}
Mariana Haragus, Eric Lombardi, and Arnd Scheel, \emph{Spectral stability of
  wave trains in the {K}awahara equation}, Journal of Mathematical Fluid
  Mechanics \textbf{8} (2006), no.~4, 482--509.

\bibitem{Johnson09nonlinear}
Mathew~A Johnson, \emph{Nonlinear stability of periodic traveling wave
  solutions of the generalized {K}orteweg--de {V}ries equation}, SIAM Journal
  on Mathematical Analysis \textbf{41} (2009), no.~5, 1921--1947.

\bibitem{JThesis}
Mathew~A. Johnson, \emph{On the stability of periodic solutions of nonlinear
  dispersive equations}, ProQuest LLC, Ann Arbor, MI, 2009, Thesis
  (Ph.D.)--University of Illinois at Urbana-Champaign. \MR{2713299}

\bibitem{J2}
\bysame, \emph{On the stability of periodic solutions of the generalized
  {B}enjamin-{B}ona-{M}ahony equation}, Phys. D \textbf{239} (2010), no.~19,
  1892--1908. \MR{2684614}

\bibitem{J88}
C.~K. R.~T. Jones, \emph{Instability of standing waves for nonlinear
  {S}chr\"{o}dinger-type equations}, Ergodic Theory Dynam. Systems \textbf{8}
  (1988), no.~Charles Conley Memorial Issue, 119--138. \MR{967634}

\bibitem{JMMP2}
Christopher K. R.~T. Jones, Robert Marangell, Peter~D. Miller, and Ram\'{o}n~G.
  Plaza, \emph{On the stability analysis of periodic sine-{G}ordon traveling
  waves}, Phys. D \textbf{251} (2013), 63--74. \MR{3039059}

\bibitem{JMMP1}
\bysame, \emph{Spectral and modulational stability of periodic wavetrains for
  the nonlinear {K}lein-{G}ordon equation}, J. Differential Equations
  \textbf{257} (2014), no.~12, 4632--4703. \MR{3268738}

\bibitem{JMMP3}
\bysame, \emph{On the spectral and modulational stability of periodic
  wavetrains for nonlinear {K}lein-{G}ordon equations}, Bull. Braz. Math. Soc.
  (N.S.) \textbf{47} (2016), no.~2, 417--429. \MR{3514411}

\bibitem{Jones1988}
C.K.R.T. Jones, \emph{Instability of standing waves for nonlinear
  {S}chr\"odinger type equations."}, Ergodic Th. \& Dynam. Sys. \textbf{8}
  (1988), 119--138.

\bibitem{kaiser}
N~Kaiser, \emph{Mean eigenvalues for simple, simply connected, compact {L}ie
  groups}, Journal of Physics A: Mathematical and General \textbf{39} (2006),
  no.~49, 15287.

\bibitem{KanoNakayama}
Kiyotsugu Kano and Toshio Nakayama, \emph{An exact traveling wave solution of
  the equation $u_t + u u_x + u_{5x}=0$}, J. Phys. Soc. Japan \textbf{50}
  (1981), 361--362.

\bibitem{Kapitula}
Todd Kapitula, \emph{The {K}rein signature, {K}rein eigenvalues, and the
  {K}rein oscillation theorem}, Indiana University Mathematics Journal (2010),
  1245--1275.

\bibitem{Kiper}
Ayse Kiper, \emph{Fourier series coefficients for powers of the {J}acobian
  elliptic functions}, Mathematics of Computation \textbf{43} (1984), 247--259.

\bibitem{KollarMiller}
Richard Koll\'{a}r and Peter~D. Miller, \emph{Graphical {K}rein signature
  theory and {E}vans--{K}rein functions}, SIAM Review \textbf{56} (2014),
  no.~1, 73--123.

\bibitem{levandosky07}
Steven Levandosky, \emph{Stability of solitary waves of a fifth-order water
  wave model}, Physica D: Nonlinear Phenomena \textbf{227} (2007), no.~2,
  162--172.

\bibitem{Panos_2023}
Meng-Meng Liu, Wen-Rong Sun, Lei Liu, PG~Kevrekidis, and Lei Wang,
  \emph{Periodic traveling waves in the $\phi^4$ model: Instability, stability,
  and localized structures}, Physical Review E \textbf{107} (2023), no.~3,
  034210.

\bibitem{MaddocksSachs93}
John~H Maddocks and Robert~L Sachs, \emph{On the stability of {KdV}
  multi-solitons}, Communications on Pure and Applied Mathematics \textbf{46}
  (1993), no.~6, 867--901.

\bibitem{MagnusWinkler1966}
W~Magnus and S~Winkler, \emph{Hill's equation},  (1966).

\bibitem{MM}
R.~Marangell and P.~D. Miller, \emph{Dynamical {H}amiltonian-{H}opf
  instabilities of periodic traveling waves in {K}lein-{G}ordon equations},
  Phys. D \textbf{308} (2015), 87--93. \MR{3377638}

\bibitem{MartelMerle01}
Yvan Martel and Frank Merle, \emph{Asymptotic stability of solitons for
  subcritical generalized {K}d{V} equations}, Archive for Rational Mechanics
  and Analysis \textbf{157} (2001), 219--254.

\bibitem{McKean77}
HP~McKean, \emph{Stability for the {K}orteweg-de {V}ries equation},
  Communications on Pure and Applied Mathematics \textbf{30} (1977), no.~3,
  347--353.

\bibitem{Natali10}
F{\'a}bio Natali, \emph{A note on the stability for {K}awahara--{K}d{V} type
  equations}, Applied Mathematics Letters \textbf{23} (2010), no.~5, 591--596.

\bibitem{Quintero03}
Jos{\'e}~R{\'a}ul Quintero, \emph{Nonlinear stability of a one-dimensional
  {B}oussinesq equation}, Journal of Dynamics and Differential Equations
  \textbf{15} (2003), no.~1, 125--142.

\bibitem{ScharfWreszinski81}
G~Scharf and W.F Wreszinski, \emph{Stability for the {K}orteweg-de {V}ries
  equation by inverse scattering theory}, Annals of Physics \textbf{134}
  (1981), no.~1, 56--75.

\bibitem{StarYak1975}
V~Starzhinskii and VA~Yakubovich, \emph{Linear differential equations with
  periodic coefficients}, London: Halsted Press, 1975.

\bibitem{TanYangPelinovsky02}
Yu~Tan, Jianke Yang, and Dmitry~E Pelinovsky, \emph{Semi-stability of embedded
  solitons in the general fifth-order {KdV} equation}, Wave Motion \textbf{36}
  (2002), no.~3, 241--255.

\bibitem{TrichtchenkoDeconinckKollar18}
Olga Trichtchenko, Bernard Deconinck, and Richard Kollár, \emph{Stability of
  periodic traveling wave solutions to the {K}awahara equation}, SIAM Journal
  on Applied Dynamical Systems \textbf{17} (2018), no.~4, 2761--2783.

\bibitem{Was18}
Wolfgang Wasow, \emph{Asymptotic expansions for ordinary differential
  equations}, Courier Dover Publications, 2018.

\bibitem{Weinstein1985}
Michael~I Weinstein, \emph{Modulational stability of ground states of nonlinear
  {S}chr{\"o}dinger equations}, SIAM Journal on Mathematical Analysis
  \textbf{16} (1985), no.~3, 472--491.

\end{thebibliography}

\end{document}